\documentclass[10pt]{amsart}
\calclayout

\usepackage{lmodern}
\usepackage{microtype}
\usepackage{amsmath,amssymb,amsthm,mathtools}
\usepackage{booktabs,array}
\usepackage[colorlinks,citecolor=black,linkcolor=black,urlcolor=black,
pdftitle={The isoperimetric problem for four-dimensional parallelohedra},
pdfauthor={Matteo Novaga and Lark Song}
]{hyperref}
\newcommand{\arXiv}[1]{arXiv:\href{https://arxiv.org/abs/#1}{#1}}
\usepackage{doi}
\newtheorem{theorem}{Theorem}[section]
\newtheorem{proposition}[theorem]{Proposition}
\newtheorem{lemma}[theorem]{Lemma}
\newtheorem{corollary}[theorem]{Corollary}

\theoremstyle{definition}

\newtheorem{remark}[theorem]{Remark}

\numberwithin{equation}{section}

\newcommand{\R}{\mathbb R}
\newcommand{\Z}{\mathbb Z}
\newcommand{\tr}{\operatorname{tr}}
\newcommand{\vol}{\operatorname{vol}}
\newcommand{\tf}{\operatorname{tf}}
\newcommand{\1}{\mathbf 1}
\newcommand{\Gr}{\operatorname{Gr}}
\newcommand{\CoGr}{\operatorname{CoGr}}
\newcommand{\Sym}{\mathrm{Sym}}
\newcommand{\Id}{\mathrm{Id}}

\title{The isoperimetric problem for four-dimensional parallelohedra}

\author{Matteo Novaga}
\address{Dipartimento di Matematica, Universit\`a di Pisa, Pisa, Italy}
\email{matteo.novaga@unipi.it}

\author{Lark Song}
\address{Department of Mathematics, University of Wisconsin, Madison, WI, USA}
\email{lark.song@wisc.edu}

\date{}
\begin{document}

\begin{abstract}
The honeycomb and Kelvin problems ask how space can be partitioned into cells of equal volume with the least surface area. We resolve this question in four dimensions for convex bodies that tile space by translations, showing that the regular $24$-cell uniquely minimizes surface area at fixed volume. We develop a framework combining an exact deficit identity, valid in every dimension, with a descent argument. As a consequence, the regular $24$-cell also minimizes surface area among such tiles containing a fixed ball. 
\end{abstract}

\maketitle
\tableofcontents

\section{Introduction}
\label{sec:introduction}

A \emph{parallelohedron} is a convex polytope that tiles Euclidean space $\R^d$ by translations.
A natural question is: among such tiles of fixed volume, which has the least surface area?
We answer this question in four dimensions, showing that the unique minimizer is the regular $24$-cell.
For a convex body $P\subset\R^d$ with $d\ge2$, write its isoperimetric quotient as
\[
  \mathcal I_d(P)=\frac{\mathcal H^{d-1}(\partial P)}{\vol_d(P)^{(d-1)/d}},
\]
and write $\mathcal I=\mathcal I_4$.
Throughout, the \emph{area} of a hypersurface in $\R^d$ is its $(d-1)$-dimensional Hausdorff measure.
For a four-dimensional polytope, the area of a facet is its three-dimensional volume.
The quotient $\mathcal I_d$ is invariant under translations, orthogonal maps, and positive dilations.
Sets related by these operations are called \emph{similar}.

\begin{theorem}
\label{thm:main}
Every four-dimensional parallelohedron $P$ satisfies
\[
  \mathcal I(P)\ge2^{11/4}.
\]
Equality holds if and only if $P$ is similar to the regular $24$-cell.
\end{theorem}

For a parallelohedron $P$ with center $c(P)$, let $\widehat P=\vol_4(P)^{-1/4}(P-c(P))$ and
\[
  \operatorname{def}(P)=\frac{\mathcal I(P)}{2^{11/4}}-1,
\]
and $C_{24}$ be a centered regular $24$-cell of volume two.
The main theorem has the following stability form.

\begin{corollary}
\label{cor:global-stability}
Let
\[
 d_{24}(P)=\inf_{U\in \mathrm{O}(4)}d_H(\widehat P,U\widehat C_{24}),
\]
where $U$ ranges over group $\mathrm{O}(4)$ of orthogonal $4\times4$ matrices, $d_H$ is Hausdorff distance.
There are constants $C,r>0$ such that every four-dimensional parallelohedron $P$ with $\operatorname{def}(P)\le1$ satisfies
\[
 d_{24}(P)\le C\sqrt{\operatorname{def}(P)}.
\]
If $d_{24}(P)<r$, then
\[
 \frac{\sqrt2}{8}d_{24}(P)^2-Cd_{24}(P)^3
 \le\operatorname{def}(P)\le C d_{24}(P)^2.
\]
The coefficient $\sqrt2/8$ is optimal as $d_{24}(P)\to0$.
Thus the deficit is locally equivalent to squared Hausdorff distance modulo similarities, with optimal exponent $1/2$ in the first estimate.
Theorem~\ref{thm:sharp-Hausdorff-coefficient} identifies all extremal first-order deformations.
\end{corollary}

The forward estimate follows directly from entropy coercivity and the exact affine remainder.
Local coordinates and second variation give the reverse estimate. A comparison on the normal cones of the regular cell gives the optimal intrinsic coefficient.

In the plane, Hales's honeycomb theorem~\cite{Hales2001Honeycomb} identifies the regular hexagonal tiling as minimizing perimeter density among equal-area partitions.
For three-dimensional parallelohedra, Hales--Song~\cite{HalesSong2026} and independently Cesaroni--Novaga~\cite{CesaroniNovaga2026} proved that the Archimedean truncated octahedron minimizes surface area at fixed volume, settling Bezdek's conjecture~\cite{Bezdek2006}.
Note that L\'angi~\cite{Langi2022} had proved the optimality of the truncated octahedron for the mean width.
In four dimensions, the regular $24$-cell honeycomb is a candidate for minimizing interfacial area among general equal-volume partitions.
Cesaroni--Novaga~\cite{CesaroniNovagaEqualCells2024} discuss this question and prove the minimality under perturbations supported in sufficiently small balls.
In Theorem~\ref{thm:main} we establish global optimality within the class of tilings by translates of a single convex tile.
The corresponding question for nonconvex fundamental domains and the four-dimensional Kelvin problem for general equal-volume partitions remain open~\cite{Thomson1887,CesaroniNovagaInradius2025}.

A direct consequence concerns convex lattice fundamental domains containing a fixed ball, where both the lattice and the volume may vary.

\begin{corollary}
\label{cor:fixed-inradius}
Let $P$ be a convex fundamental domain for a lattice in $\R^4$.
If $P$ contains a Euclidean ball of radius $r>0$, then
\[
  \mathcal H^3(\partial P)\ge32r^3.
\]
Equality holds if and only if $\overline P$ is congruent to the regular $24$-cell of inradius $r$.
\end{corollary}

This bound follows from Theorem~\ref{thm:main} and the pyramid formula.
It proves the convex case of the fixed-inradius conjecture in~\cite{CesaroniNovagaInradius2025}.
The analogous three-dimensional bound was proved in~\cite{CesaroniNovagaInradius2026}.

By Delaunay's theorem, every four-dimensional parallelohedron is affinely equivalent to a lattice Voronoi cell.
This part of Voronoi's conjecture is now known through dimension five~\cite{Garber2025}.
The classification reduces the global problem to three closed cones of lattice metrics, including their positive-definite boundary forms.
A further volume-preserving affine map must still be allowed, since it changes surface area and need not preserve the Euclidean Voronoi property.
We therefore estimate $\mathcal I(LV_H)$ jointly in the metric $H$ and every $L$ with $|\det L|=1$.

The principal cone is separated from the minimum by a strict gap.
Its cells are affine images of zonotopes with generator directions indexed by $K_5$.
A determinant estimate, stationarity in the generator lengths, and the sharp three-dimensional inequality for projections give the required bound.
The two nonprincipal cones share the $D_4$ metric, whose Voronoi cell is the regular $24$-cell.
This cell is not a zonotope, and changes of metric can create additional facets.

The mechanism is a balance between directional distortion and facet creation.
At fixed volume, unequal directional stretches incur a nonnegative cost through convexity.
New facets change the relation between surface area and enclosed volume, so the proof subtracts a weighted allowance for their contribution.
The decisive estimate shows that the distortion cost dominates this allowance.
The exact deficit identity then records further nonnegative costs of the affine realization.
This explains why equality forces the regular cell and why a small area excess controls its geometric distortion.

Cofactor balance links the facet weights in the surface formula to the metric $H$.
In either nonprincipal cone it yields
\[
 \mathcal I(LV_H)\ge2^{11/4}\exp(F/8),\quad\text{where }F=E-\theta\log2.
\]
Here $E$ measures directional variation of $H$ and $\theta/4$ is the volume fraction of the pyramids over the new facets.
The compensation records their different surface normalization.

The difficult step is to prove $F\ge0$ despite the contribution of the new facets.
We differentiate $(\det H)F$ in a largest retained weight and reduce its positivity to bounds for transverse entropy and a cofactor mean.
In four dimensions, an entropy inequality for spanning trees of $K_4$ supplies the first bound, and a common sign-cube model for the two nonprincipal cones supplies the second.
Integration along a capped path of weights then gives strictness away from the regular core and an explicit uniform gap on the zero-core boundary.
Theorem~\ref{thm:portable-global-comparison} records the hypotheses under which this descent argument applies.

The exact deficit identity of Theorem~\ref{lem:balanced-entropy-bridge} is algebraic and holds in arbitrary dimension.
Positivity of the compensated entropy additionally requires the geometric estimates just described.
In the two nonprincipal cones, entropy coercivity and the identity's affine remainder give the direct geometric estimate of Corollary~\ref{cor:entropy-geometric-control}, avoiding a compactness argument to enter the small-deficit stability neighborhood.

The sharp local rate comes from a joint second variation, uniform across changes of Voronoi type.
The distinction between metric and affine perturbations already appears in the three-dimensional local problem~\cite{CesaroniNovagaLocal2026,Song2026Local}.
Cofactor balance determines its mixed terms without differentiating individual facet areas.
At the $24$-cell, new facets contribute only at cubic order.
A sharp comparison of the quadratic form with support displacements gives the asymptotically optimal local Hausdorff coefficient $\sqrt2/8$, with a cubic remainder.
Its unique extremal mode, up to regular-cell symmetries and positive scale, combines metric deformation with affine strain.
The same theorem gives a quantitative distance bound from this mode when the normalized area cost is nearly optimal (Theorem~\ref{thm:sharp-Hausdorff-coefficient}).

We write $\Id$ for the identity and $\Pi_v=vv^{\mathsf T}/\|v\|^2$ for projection onto the line spanned by $v\ne0$.
Determinants are denoted by $D_H=\det H$, and relative entropy by $\operatorname{KL}$.
The regular reference cell is always $C_{24}$, and constants in estimates are denoted by $c,C$.
For real matrices we use the Frobenius norm $\|A\|^2=\tr(A^{\mathsf T}A)$.
For symmetric matrices the associated pairing is $\langle A,B\rangle=\tr(AB)$.
We write $\Sym_0(\R^d)$ for the traceless symmetric matrices and $\tf A=A-(\tr A){\Id}/d$ for the traceless part.
Sections~\ref{sec:reduction}--\ref{sec:mean-bounds} establish the global comparison, from Voronoi classification and cofactor balance to entropy descent and its mean bounds.
Section~\ref{sec:local-variation} establishes sharp stability of the $24$-cell and identifies its extremal deformation.

\section{Voronoi classification and the principal cone}\label{sec:reduction}

By the Venkov--McMullen theorem~\cite{Venkov1954,McMullen1980}, every convex translative tile is a polytope admitting a face-to-face lattice tiling.
In particular, the closure of a convex lattice fundamental domain is a parallelohedron.

In Voronoi's reduction theory~\cite{Voronoi1908}, an $L$-domain consists of positive-definite forms with the same Delaunay decomposition of $\Z^4$, and is \emph{primitive} when that decomposition is simplicial.
Its closure is taken in the positive-semidefinite cone.
Positive-definite boundary forms still define full-dimensional cells, while singular forms enter only as limits.

For a positive-definite form ${H_{\mathbb Z}}={B_\Lambda}^{\mathsf T}{B_\Lambda}$ on $\Z^4$ define
\[
  \operatorname{Vor}({H_{\mathbb Z}})=\{x\in\R^4:2x^{\mathsf T}{H_{\mathbb Z}}z\le z^{\mathsf T}{H_{\mathbb Z}}z\ \text{for all }z\in\Z^4\}.
\]
The Voronoi cell of the lattice ${B_\Lambda}\Z^4$ is ${B_\Lambda}\operatorname{Vor}({H_{\mathbb Z}})$, of volume $\sqrt{\det {H_{\mathbb Z}}}$.
To use Minkowski addition within a single $L$-domain, we pass to the cell in dual coordinates
\begin{equation*}
\operatorname{Vor}_{\rm dual}(H_{\mathbb Z})
 =2H_{\mathbb Z}\operatorname{Vor}(H_{\mathbb Z})
 =\{y:z^{\mathsf T}y\le z^{\mathsf T}H_{\mathbb Z}z \text{ for all }z\in\Z^4\}.
\end{equation*}
\begin{lemma}
\label{lem:minkowski}
If ${H_{\mathbb Z,1}},{H_{\mathbb Z,2}}$ lie in the closure of one $L$-domain, then
\[
  {\operatorname{Vor}_{\rm dual}}({H_{\mathbb Z,1}}+{H_{\mathbb Z,2}})={\operatorname{Vor}_{\rm dual}}({H_{\mathbb Z,1}})+{\operatorname{Vor}_{\rm dual}}({H_{\mathbb Z,2}}).
\]
If $v\in\Z^4$ is primitive, then ${\operatorname{Vor}_{\rm dual}}(vv^{\mathsf T})=[-v,v]$.
\end{lemma}
\begin{proof}
Within an $L$-domain the normal fan is fixed, and the vertex equations $z_i^{\mathsf T}y=z_i^{\mathsf T}H_{\mathbb Z}z_i$ depend linearly on $H_{\mathbb Z}$.
Thus support functions add, and continuity extends the identity to the boundary.
For $H_{\mathbb Z}=vv^{\mathsf T}$, the inequalities with $z\perp v$ force $y=sv$.
Primitivity makes $v^{\mathsf T}z$ range over $\Z$, giving $|s|\le1$.
\end{proof}

\begin{theorem}
\label{thm:classification}
Every four-dimensional parallelohedron is affinely equivalent to a lattice Voronoi cell.
Up to arithmetic equivalence, every positive-definite form belongs to the closure of one of three primitive $L$-domains, namely the principal cone (Selling's cone) and two nonprincipal cones, called types II and III.
\end{theorem}

\begin{proof}
The affine statement is Delaunay's theorem~\cite{Delaunay1929I,Delaunay1929II}, also discussed in~\cite{GarberMagazinov2022}.
The three primitive types are listed in~\cite[Table~2]{DutourGrishukhin2009}.
Section~6.2 of that reference identifies the principal and two nonprincipal cones and their graphic or cographic descriptions.
\end{proof}

The lattice $D_4=\{x\in\Z^4:x_1+x_2+x_3+x_4\equiv0\pmod2\}$ fixes the normalization of the sharp constant and equality case.
\begin{lemma}
\label{lem:reference-cell}
The Voronoi cell of $D_4$ is the regular $24$-cell
\[
  C_{24}=\operatorname{conv}\Bigl(\{\pm e_i:1\le i\le4\} \cup\{\frac12(\pm1,\pm1,\pm1,\pm1)\}\Bigr).
\]
It has volume $2$, inradius $1/\sqrt2$, and $24$ regular octahedral facets, each of area $\sqrt2/3$.
Consequently $\mathcal H^3(\partial C_{24})=8\sqrt2$ and $\mathcal I(C_{24})=2^{11/4}$.
\end{lemma}
\begin{proof}
The $D_4$ inequalities $|y_i|+|y_j|\le1$ have the displayed vertices.
Each root facet is a regular octahedron of circumradius and support distance $1/\sqrt2$, hence area $\sqrt2/3$.
Summing the $24$ facet areas and using the pyramid formula gives the remaining quantities.
\end{proof}

Write $\Gr(G)$ for a graphic vector configuration, whose independent sets are the forests of $G$.
A realization is \emph{unimodular} if its nonzero maximal determinants have the same absolute value.

For each oriented edge $\epsilon$ of $K_5$, let $\chi_\epsilon\in\Z^4$ be its incidence vector with one fixed vertex coordinate deleted.

\begin{proposition}
\label{prop:principal-graphic}
Every form in the closed principal cone has the form
\[
  {H_{\mathbb Z}}=\sum_{\epsilon\in E(K_5)}{x_\epsilon} \chi_\epsilon\chi_\epsilon^{\mathsf T},\quad\text{with } {x_\epsilon}\ge0.
\]
If ${H_{\mathbb Z}}>0$, its Voronoi cell is affinely equivalent to $\sum_\epsilon{x_\epsilon}[-\chi_\epsilon,\chi_\epsilon]$.
After zero parameters are discarded, the generating system is therefore a subsystem of $\Gr(K_5)$.
\end{proposition}
\begin{proof}
Apply Lemma~\ref{lem:minkowski} to the Selling parametrization~\cite{Selling1874,DutourGrishukhin2009}, then undo the invertible dual-coordinate map.
\end{proof}

The next lemma excludes degenerate minimizing sequences in the proof of Theorem~\ref{thm:zonotope-gap}.

\begin{lemma}
\label{lem:degeneracy}
Let $\Omega\subset\R^4$ be a convex body of diameter ${d_\Omega}$, and write ${\operatorname{area}}(\Omega)=\mathcal H^3(\partial\Omega)$ and $\omega_3=4\pi/3$.
Then
\[
  {\operatorname{area}}(\Omega)^3\ge\frac{27\omega_3}{4}\vol(\Omega)^2{d_\Omega}.
\]
\end{lemma}
\begin{proof}
Choose diameter endpoints $x,y$, and set $\xi(z)=(z-x)\cdot(y-x)/d_\Omega$, $\Omega_c=\Omega\cap\{\xi=c\}$, and $\operatorname{slice}(c)=\mathcal H^3(\Omega_c)$.
Joining a slice of maximal area $\operatorname{slice}_*$ to $x,y$ gives two cones in $\Omega$, hence $\vol(\Omega)\ge\operatorname{slice}_*d_\Omega/4$.
Since $\xi$ is $1$-Lipschitz and the relative slice boundaries lie in $\partial\Omega$, coarea and the three-dimensional isoperimetric inequality give
\[
  \operatorname{area}(\Omega)\ge3\omega_3^{1/3}\int_0^{d_\Omega}\operatorname{slice}(c)^{2/3}\,dc
  \ge3\omega_3^{1/3}\frac{\vol(\Omega)}{\operatorname{slice}_*^{1/3}}
  \ge3\omega_3^{1/3}\vol(\Omega)^{2/3}\Bigl(\frac{d_\Omega}4\Bigr)^{1/3}.
\]
Cubing proves the claim.
\end{proof}

Theorem~\ref{thm:zonotope-gap} below treats the principal cone.
The generator-length estimate is the four-dimensional instance of the weighted simplex inequality of Hales--Song~\cite[Theorem~2.2]{HalesSong2026}.
We reproduce its determinant argument below.
The remaining input is the sharp three-dimensional theorem~\cite[Theorem~1.1]{CesaroniNovaga2026}.

\begin{theorem}
\label{thm:zonotope-gap}
Let ${P}$ be a full-dimensional zonotopal parallelohedron which, up to translation, has a representation
\[
  P=\sum_{\epsilon\in E_0}c_\epsilon[-L_{\rm gen}\chi_\epsilon/2,L_{\rm gen}\chi_\epsilon/2],\quad
  \text{with }\; E_0\subset E(K_5),\quad c_\epsilon>0,\quad L_{\rm gen}\in \mathrm{GL}(4,\R).
\]
Then
\begin{equation*}
\mathcal I({P})^4\ge4\cdot 5^{1/6}(1+2\sqrt3)^4>2048.
\end{equation*}
Since the hypothesis is invariant under linear maps, the bound holds for every affine image of ${P}$.
In particular, every affine image of the Voronoi cell of a positive-definite form in the closed principal cone satisfies $\mathcal I>2^{11/4}$.
\end{theorem}
\begin{proof}
Enumerate $E_0=\{\epsilon_1,\ldots,\epsilon_{n_0}\}$ and let $v_i=L_{\rm gen}\chi_{\epsilon_i}$, $r_i=\|v_i\|$, $n_i=v_i/r_i$, and $\ell_i=c_{\epsilon_i}r_i$.
By homogeneity normalize $\sum_i\ell_i=1$, so $P=\sum_i\ell_i[-n_i/2,n_i/2]$.
The common nonzero basis determinant of $(v_i)$ is $D_{\rm gen}=|\det L_{\rm gen}|$, and its bases $\mathcal B$ are the spanning trees on $E_0$.
Set $\ell_J=\prod_{i\in J}\ell_i$, $Z_{\rm tree}=\sum_{J\in\mathcal B}\ell_J$, and $\Sigma=\sum_i\ell_i n_in_i^{\mathsf T}$.
The volume formula, Cauchy--Binet, and Cauchy--Schwarz give
\[
 (\vol P)^2
 =\left(\sum_{J\in\mathcal B}\frac{\ell_JD_{\rm gen}}{\prod_{i\in J}r_i}\right)^2
 \le Z_{\rm tree}\det\Sigma\le\frac{Z_{\rm tree}}{256},
\]
where the last step uses $\tr\Sigma=1$.
Extended by zero to $K_5$, the weights give a Laplacian of trace $2$.
The matrix-tree theorem and AM--GM yield $Z_{\rm tree}=\frac15\prod_{j=1}^4\lambda_j\le1/80$ for its positive eigenvalues.
Hence
\begin{equation*}
 \vol(P)\le\frac1{64\sqrt5},\quad\text{with }\sum_i\ell_i=1.
\end{equation*}

Positive length changes preserve tiling, and on the simplex, $\operatorname{diam}P\ge\max_i\ell_i\ge1/n_0$.
Lemma~\ref{lem:degeneracy} therefore gives $\mathcal I(P)\to\infty$ as $\vol(P)\to0$ and a full-dimensional minimizer $P_{\rm min}$, stationary in every active length.
Its active projections also tile.
To see this, quotient a unimodular basis containing the generator and combine parallel directions.
Every independent quotient triple extends to a basis~\cite{McMullen1975,DezaGrishukhin2004}.
Orthogonal projection differs by an invertible linear map.

Write $v_{\rm min},s_{\rm min}$ for its volume and area, and $v_{{\rm proj},i},s_{{\rm proj},i}$ for those of its projection onto $n_i^\perp$.
Extrusion gives $\partial_i v_{\rm min}=v_{{\rm proj},i}$ and $\partial_i s_{\rm min}=s_{{\rm proj},i}$.
Degree-zero homogeneity makes the Lagrange multiplier vanish.
Stationarity and the sharp three-dimensional bound~\cite[Theorem~1.1]{CesaroniNovaga2026} give
\[
 s_{{\rm proj},i}=\frac{3s_{\rm min}}{4v_{\rm min}}v_{{\rm proj},i}
 \ge c_3v_{{\rm proj},i}^{2/3},
 \quad\text{where }c_3=\frac{3(1+2\sqrt3)}{4^{2/3}}.
\]
Thus $v_{{\rm proj},i}\ge(4c_3v_{\rm min}/(3s_{\rm min}))^3$.
Sum with weights $\ell_i$, using $\sum_i\ell_i=1$ and $\sum_i\ell_i v_{{\rm proj},i}=4v_{\rm min}$, to obtain $\mathcal I(P_{\rm min})^3\ge(1+2\sqrt3)^3/v_{\rm min}^{1/4}$.
Raising to $4/3$ and using the normalized volume bound proves the claim.
Finally, $5^{1/6}>13/10$, $\sqrt3>19/11$, and $(1+2\sqrt3)^4=217+104\sqrt3$ put it above $113438/55>2048$.
\end{proof}
\section{Cofactor balance and compensated entropy}\label{sec:entropy-framework}

We separate the comparison into an algebraic statement valid in every dimension and a geometric estimate for the admissible four-dimensional configurations.
The facet lengths determine the entropy penalty, and rank-one variation then identifies the contribution of projected cofactors and the remaining exchange cost.

\subsection{Cofactor balance and the exact surface formula}

Let $\Lambda\subset\R^d$ be a lattice of covolume $v_\Lambda$, set
\[
  \mathcal C_H=\{x:v\cdot x\le v^{\mathsf T}Hv/2\text{ for all }v\in\Lambda\}
  \quad\text{and}\quad
    V_H=H^{-1/2}\mathcal C_H
\]
for $H>0$.
Then $V_H=\operatorname{Vor}(H^{1/2}\Lambda)$ has volume $v_\Lambda\sqrt{D_H}$, where $D_H=\det H$.
Choose one vector $v$ from each opposite facet pair, and let $b_e$ be the area of its facet in $\mathcal C_H$.
Define
\[
  \ell_e=\|v\|,\qquad e=v/\ell_e,\qquad w_e=\ell_e b_e/v_\Lambda.
\]
Write $\Pi_e=ee^{\mathsf T}$, $h_e=e^{\mathsf T}He$, and $n_e=H^{1/2}e/\sqrt{h_e}$.
Give absent facets weight zero.

\begin{lemma}
\label{lem:universal-cofactors}
For every $H>0$,
\[
  \sum_e w_e\Pi_e=\operatorname{adj}H.
\]
For a reference length $\ell_0>0$, set $\mathcal I_0=2dv_\Lambda^{1/d}/\ell_0$ and $\kappa_e=(\ell_e/\ell_0)^2$.
Then
\begin{equation}
\label{eq:universal-surface-formula}
\frac{\mathcal I_d(LV_H)}{\mathcal I_0}
=\frac{D_H^{-1+1/(2d)}}d\sum_e\kappa_e^{-1/2} w_e\|L^{-\mathsf T}H^{1/2}e\|,
\quad\;\text{with }|\det L|=1.
\end{equation}
\end{lemma}
\begin{proof}
Reflection about $H^{1/2}v/2$ interchanges adjacent Voronoi cells, so the facet in $\mathcal C_H$ is centered at $Hv/2$.
The divergence identity, with outward unit normal $n(x)$, gives
\[
  \sum_e b_e\ell_e H\Pi_e
  =\int_{\partial\mathcal C_H}x\otimes n(x)\,dS=v_\Lambda D_H \Id.
\]
Multiplication by $H^{-1}$ proves cofactor balance.
The maps $H^{-1/2}$ and $L$ multiply facet area by $\sqrt{h_e}/\sqrt{D_H}$ and $\|L^{-\mathsf T}n_e\|$, respectively.
Summing opposite pairs and normalizing by volume gives the surface formula.
\end{proof}

For a positive-definite matrix $A$ and a unit vector $e$, write
\[
  \operatorname{ent}_e(A)=e^{\mathsf T}A\log A\,e-(e^{\mathsf T}Ae)\log(e^{\mathsf T}Ae)
  \quad\text{and}\quad
  \operatorname{var}_e(A)=e^{\mathsf T}A^2e-(e^{\mathsf T}Ae)^2.
\]
These nonnegative quantities extend continuously to positive semidefinite matrices.

\subsection{Cofactor entropy and the exact deficit}

The following statement is algebraic.
In an application, its weights and attenuation factors must first be identified with the actual facet data.
In particular, cofactor balance alone does not imply that the compensated entropy below is nonnegative.

\begin{theorem}
\label{lem:balanced-entropy-bridge}
Let $d\ge2$, let $H>0$ be a real symmetric $d\times d$ matrix, and let $\mathcal N$ be a finite indexed family of unit vectors with nonnegative weights satisfying cofactor balance.
Write
\[
 {D_H}:=\det H\quad\text{and}\quad \sum_{e\in\mathcal N}w_e\Pi_e=\operatorname{adj}H.
\]
Let $\kappa_e>0$, and define
\[
 h_e=e^{\mathsf T}He,\quad
 \lambda_e=\frac{w_eh_e}{{D_H}},\quad
 E=\frac1{D_H}\sum_e w_e\operatorname{ent}_e(H),\quad
 F=E-\sum_e\lambda_e\log\kappa_e,
\]
\[
 W_\kappa=\sum_e\frac{w_e}{\sqrt{\kappa_eh_e}}\Pi_e ,
 \quad\text{and}\quad j_\kappa=\frac{\det W_\kappa}{{D_H}^{d-3/2}}.
\]
Write $\mathcal B$ for the $d$-element bases of positive-weight vectors in $\mathcal N$, and set
\[
 \mu(B)=\frac{c_B}{{D_H}^{d-1}}\quad\text{and}\quad
 \nu(B)=\frac{c_B\prod_{e\in B}(\kappa_eh_e)^{-1/2}}
                    {\det W_\kappa},
\]
where $c_B=\det[e:e\in B]^2\prod_{e\in B}w_e$.
These are probability laws with the same support, and we use $\operatorname{KL}(\mu\|\nu)=\sum_B\mu(B)\log(\mu(B)/\nu(B))$ for their relative entropy.
Then
\begin{equation*}
 E=\log {D_H}-\sum_e\lambda_e\log h_e\ge0
 \quad\text{and}\quad
 \log j_\kappa=\frac F2+\operatorname{KL}(\mu\|\nu).
\end{equation*}
For $|\det L|=1$, define
\[
 \operatorname{surf}_H(L)=\frac{{D_H}^{-1+1/(2d)}}d
       \sum_e\kappa_e^{-1/2}w_e
                         \|L^{-\mathsf T}H^{1/2}e\|.
\]
Write $L^{-\mathsf T}=US$ with $S>0$ and $\det S=1$, and define
\[
 n_e=\frac{H^{1/2}e}{\sqrt{h_e}},\quad
 w_e^{\mathrm{surf}}=\frac{w_e\sqrt{h_e}}{\sqrt{\kappa_e}},\quad
 \widehat W_\kappa=\sum_e w_e^{\mathrm{surf}}\Pi_{n_e}=H^{1/2}W_\kappa H^{1/2},
\]
\[
 T=\frac{\widehat W_\kappa^{1/2}S\widehat W_\kappa^{1/2}}
                {(\det\widehat W_\kappa)^{1/d}},\quad\text{and}\quad
 g_{\mathrm{ang}}=\frac{\sum_e w_e^{\mathrm{surf}}(\|Sn_e\|-n_e^{\mathsf T}Sn_e)}
                  {\operatorname{tr}(S\widehat W_\kappa)}.
\]
Then $\det T=1$, $g_{\mathrm{ang}}\ge0$, and the exact identity
\begin{equation*}
 \log\operatorname{surf}_H(L)
 =\frac{F}{2d}+\frac{\operatorname{KL}(\mu\|\nu)}d
       +\log\frac{\operatorname{tr}T}{d}+\log(1+g_{\mathrm{ang}})
\end{equation*}
decomposes the surface deficit into compensated entropy and three nonnegative terms.
In particular,
\[
 \operatorname{surf}_H(L)\ge j_\kappa^{1/d}\ge\exp(F/(2d)).
\]
\end{theorem}

\begin{proof}
Tracing cofactor balance against $H$ and $H\log H$ gives
\[
 \sum_e\lambda_e=d\quad\text{and}\quad
 \sum_e w_e e^{\mathsf T}H\log H\,e={D_H}\log {D_H}.
\]
Spectral Jensen gives $E\ge0$, and Cauchy--Binet normalizes $\mu,\nu$ since $\operatorname{adj}H>0$.
Differentiating the weights independently before imposing balance gives
\[
 \Pr_\mu(e\in B)
 =w_e\frac{\partial}{\partial w_e}
       \log\det\Bigl(\sum_f w_f\Pi_f\Bigr)
 =w_e e^{\mathsf T}(\operatorname{adj}H)^{-1}e
 =\lambda_e.
\]
Hence
\[
 \operatorname{KL}(\mu\|\nu)
 =\log\det W_\kappa-(d-1)\log {D_H}
                  +\frac12\sum_e\lambda_e\log(\kappa_eh_e)
 =\log j_\kappa-\frac F2,
\]
which proves the entropy identities.

For unit $n$, Cauchy--Schwarz gives $\|Sn\|\ge n^{\mathsf T}Sn>0$, hence $g_{\mathrm{ang}}\ge0$.
Since $\det\widehat W_\kappa=D_H\det W_\kappa$, the surface formula becomes
\[
 \operatorname{surf}_H(L)
   =j_\kappa^{1/d}\frac{\operatorname{tr}T}{d}(1+g_{\mathrm{ang}}).
\]
Taking logarithms and using $\operatorname{tr}T\ge d$ proves the decomposition and bounds.
\end{proof}

\begin{remark}
The three remainders in the exact deficit decomposition vanish precisely under the following respective conditions:
\[
 \prod_{e\in B}\kappa_eh_e\ \hbox{is independent of }B\in\mathcal B,
 \quad\; T=\Id,
 \quad\; Sn_e\in\mathbb Rn_e \;\text{ with }\; w_e>0.
\]
The first follows from $\operatorname{KL}(\mu\|\nu)=0$ if and only if $\mu=\nu$, while the last is equality in Cauchy--Schwarz.
Moreover $E=0$ precisely when every positive-weight $e$ is an eigenvector of $H$.
If their graph of nonorthogonal pairs is connected, this forces $H$ to be scalar.

If $H=c\Id$ and $\kappa_e=1$ whenever $w_e>0$, cofactor balance gives $E=F=\log j_\kappa=0$ and $\widehat W_\kappa=c^{d-1/2}\Id$.
Then $T=S$, so $\operatorname{surf}_H(L)\ge1$, with equality precisely when $S=\Id$, equivalently when $L$ is orthogonal.

For the facet marks and reference $\mathcal I_0$ in Lemma~\ref{lem:universal-cofactors}, $\operatorname{surf}_H(L)=\mathcal I_d(LV_H)/\mathcal I_0$.
\end{remark}

\begin{corollary}
\label{cor:exact-affine-coercivity}
In Theorem~\ref{lem:balanced-entropy-bridge}, let $r_{\mathrm{aff}}=\log\operatorname{surf}_H(L)-F/(2d)\ge0$.
For a constant $C_d$ depending only on $d$,
\begin{equation*}
 \|\log T\| \le C_d(\sqrt{r_{\mathrm{aff}}}+{r_{\mathrm{aff}}}) \quad\text{and}\quad
 \|\log S+\tf(\log\widehat W_\kappa)\|_{\mathrm{op}}
 \le\|\log T\|_{\mathrm{op}}.
\end{equation*}
Thus the trace remainder controls the affine strain about the determinant-normalized inverse facet matrix, for every volume-preserving $L$.
If $F\ge0$, the right side can instead be bounded using $\log\operatorname{surf}_H(L)$.
\end{corollary}
\begin{proof}
Let $z_i$ be the eigenvalue logarithms of $T$ and $r=\log(\operatorname{tr}T/d)$.
Then $\sum_i z_i=0$ and $0\le r\le r_{\mathrm{aff}}$.
For $m=\max_i z_i$, Jensen, the linear and quadratic exponential bounds give
\[
 de^r\ge e^m+(d-1)e^{-m/(d-1)}\ge d+\frac{m^2}{2}
 \quad\text{and}\quad m\le r+\log d.
\]
Together with $\sum_i z_i^2\le d(d-1)m^2$, these bounds give $\|\log T\|\le C_d(\sqrt r+r)$, using $r\le1$ and $r\ge1$ separately.

Set $S_0=(\det\widehat W_\kappa)^{1/d}\widehat W_\kappa^{-1}$ and $a=\|\log T\|_{\mathrm{op}}$.
Since $T=S_0^{-1/2}SS_0^{-1/2}$, $e^{-a}S_0\le S\le e^aS_0$.
The logarithm preserves matrix order by its integral representation and the order reversal of inversion.
Thus $-a\Id\le\log S-\log S_0\le a\Id$, where $\log S_0=-\tf(\log\widehat W_\kappa)$.
\end{proof}

\subsection{Marked entropy under a rank-one change}

The next identity requires both cofactor balance and the specified dependence of the individual weights on the chosen rank-one parameter.
These hypotheses are satisfied by the explicit facet cofactors in the two nonprincipal cones.

\begin{lemma}\label{lem:marked-rank-one}
Fix a unit vector $a$, a positive semidefinite $H_-$, and $H(R)=H_-+R\Pi_a>0$ on an open interval of positive $R$.
Suppose a fixed finite family $\mathcal N$ of unit vectors has nonnegative cofactor weights satisfying balance, with $w_e=R\omega_e+\zeta_e$ for $e\perp a$ and $w_e$ independent of $R$ otherwise.
Here $\omega_e,\zeta_e\ge0$ are fixed.
Fix marks $\kappa_e>0$ and use the entropy notation of Theorem~\ref{lem:balanced-entropy-bridge}.
In the orthogonal splitting $\mathbb Ra\oplus a^\perp$, write
\[
 {H_-}=\begin{pmatrix}g_{00}&g^{\mathsf T}\\g&{H_\perp}\end{pmatrix},
 \quad {D_\perp}=\det {H_\perp},\quad
 {h_{\mathrm S}}=g_{00}-g^{\mathsf T}{H_\perp}^{-1}g.
\]
Thus ${H_\perp}>0$, ${h_{\mathrm S}}\ge0$, $\det {H_-}={D_\perp}{h_{\mathrm S}}$, and ${D_H}={D_\perp}(R+{h_{\mathrm S}})$.
For $e\not\perp a$, set
\[
 \alpha_e=(e\cdot a)^2,\qquad
 \beta_e=\frac{e^{\mathsf T}{H_-}e}{\alpha_e},\qquad
 p_e^{\mathrm{dir}}=\frac{w_e\alpha_e}{{D_\perp}}.
\]
Then $p^{\mathrm{dir}}$ is a probability law and $\beta_e\ge{h_{\mathrm S}}$.
For perpendicular normals set $h_e=e^{\mathsf T}{H_\perp}e$ and
\[
 E_\perp=\sum_{e\perp a}\omega_e\operatorname{ent}_e({H_\perp})
 \quad\text{and}\quad
 E_{\mathrm{pen},\perp}=\sum_{e\perp a}\omega_eh_e\log\kappa_e.
\]
The compensated quantity $F_{\mathrm{det}}(R)={D_H}(R)F(R)$ satisfies the exact identity
\[
 F_{\mathrm{det}}'(R)
 =E_\perp-E_{\mathrm{pen},\perp}
   -{D_\perp}\sum_{e\not\perp a}p_e^{\mathrm{dir}}\log(\alpha_e\kappa_e)
 -{D_\perp}\sum_{e\not\perp a}p_e^{\mathrm{dir}}
                    \log\frac{R+\beta_e}{R+{h_{\mathrm S}}}.
\]
If $\alpha_e\kappa_e=e^{-b}$ on positive-weight nonperpendicular normals, the angular term is $D_\perp b$.

Define
\[
 {N_{\perp,0}}=\sum_{e\perp a}\zeta_eh_e\quad\text{and}\quad
 Q=R{D_\perp}-(d-2)\det {H_-}+{N_{\perp,0}}.
\]
One always has $Q\le {D_H}$ and
\[
 \sum_{e\not\perp a}p_e^{\mathrm{dir}}
       \frac{R+\beta_e}{R+{h_{\mathrm S}}}=2-\frac Q{D_H}.
\]
Thus Jensen bounds the final logarithmic average by $\log(2-Q/D_H)$.
No lower bound for $Q$ is asserted.
\end{lemma}

\begin{proof}
Since $H_\perp>0$, Schur completion gives the determinant formulas and $\beta_e\ge h_{\mathrm S}$.
Differentiating cofactor balance gives
\[
 \sum_{e\perp a}\omega_e\Pi_e=\operatorname{adj}{H_\perp}
 \quad\text{on }\; a^\perp.
\]
The $a,a$ entry gives $\sum_{e\not\perp a}w_e\alpha_e=D_\perp$, normalizing $p_e^{\mathrm{dir}}$.
Transverse traces give
\[
 \sum_{e\perp a}\omega_eh_e=(d-1){D_\perp}\quad\text{and}\quad
 E_\perp={D_\perp}\log {D_\perp}-\sum_{e\perp a}\omega_eh_e\log h_e.
\]

Differentiate $F_{\mathrm{det}}=D_H\log D_H-\sum_e w_eh_e\log(\kappa_eh_e)$ using $D_H'=D_\perp$ and $h'_e=0$ or $\alpha_e$ according as $e\perp a$ or not.
The constant terms cancel, leaving
\[
 F_{\mathrm{det}}'
 =D_\perp\log D_H-\sum_{e\perp a}\omega_eh_e\log(\kappa_eh_e)
 -D_\perp\sum_{e\not\perp a}p_e^{\mathrm{dir}}
                       \log[\kappa_e\alpha_e(R+\beta_e)].
\]
Substitute the transverse trace identity and $D_H=D_\perp(R+h_{\mathrm S})$ to obtain the marked rank-one identity.

Tracing balance against $H$ and subtracting the transverse contribution yields
\[
 D_\perp\sum_{e\not\perp a}p_e^{\mathrm{dir}}(\beta_e-h_{\mathrm S})
 =(d-1)D_\perp h_{\mathrm S}-N_{\perp,0}=D_H-Q\ge0.
\]
This gives the mean ratio and $Q\le D_H$.
Jensen is an equality precisely when $\beta_e$ is constant on positive-weight nonperpendicular normals.
\end{proof}

\begin{remark}
Zero weights are omitted from probability sums.
The proof requires no inverse or logarithm of $H_-$.
If it is singular, $h_{\mathrm S}=0$ and the last display forces $N_{\perp,0}=0$ and $Q=D_H$.
Since $\beta_e\ge h_{\mathrm S}$, every positive-weight term then has $\beta_e=0$, so the logarithmic correction vanishes.
\end{remark}

The next criterion gives a global comparison from bounds for transverse entropy and a cofactor mean along largest weights, without assuming $F\ge0$.

\begin{theorem}
\label{thm:portable-global-comparison}
Let $d\ge2$, let $a_1,\ldots,a_N\in\R^d$ be unit vectors, and define
\[
 H(t,x)=t\Id+\sum_{i=1}^Nx_i\Pi_{a_i}>0,
 \quad\text{with}\quad t,x_i\ge0.
\]
Suppose a fixed finite family of unit vectors has nonnegative cofactor weights $w_e(t,x)$ and fixed marks $\kappa_e>0$ satisfying balance.
Assume the weights are continuous, homogeneous of degree $d-1$, $C^1$ in positive $x$-coordinates on each fixed-support face, and satisfy Lemma~\ref{lem:marked-rank-one} in each positive $x_i$.
Let $F$ be the corresponding compensated entropy, and require $w_e(t,0)>0\Rightarrow\kappa_e=1$.
For each largest positive coordinate $R=x_i$, define the transverse budget
\[
 b_i^{\mathrm{ent}}=\frac{E_\perp-E_{\mathrm{pen},\perp}}{{D_\perp}}
       -\sum_{e\not\perp a_i}p_e^{\mathrm{dir}}\log(\alpha_e\kappa_e).
\]
Let $b_i^{\mathrm{ent}}\ge\underline b_i^{\mathrm{ent}}$ and $Q/D_H\ge\underline q_i$ at every such coordinate.
If either test
\begin{equation*}
 \underline b_i^{\mathrm{ent}}\ge\log(2-\underline q_i)
 \quad\text{or}\quad
 \underline b_i^{\mathrm{ent}}-\log2+\frac{\underline q_i}{2}\ge0
\end{equation*}
holds at each point and coordinate, then the compensated entropy $F$ is nonnegative, including on the positive-definite zero-core boundary.
If the chosen test is strict at every such point and coordinate, $F=0$ precisely at $x=0$.

If the cofactor family is the actual facet data of a lattice of covolume $v_\Lambda$, with
\[
 \kappa_e=(\ell_e/\ell_0)^2
 \quad\text{and}\quad \mathcal I_0=2dv_\Lambda^{1/d}/\ell_0,
\]
then $\mathcal I_d(LV_H)\ge \mathcal I_0$ for every $|\det L|=1$.
Under the strict tests, equality holds precisely at the core with orthogonal $L$.
If, on $t=0$, the left-minus-right margin in either chosen test is uniformly at least $\varepsilon>0$, then $F\ge\varepsilon$.
For actual facet data this yields
\begin{equation*}
 \mathcal I_d(LV_H)\ge \mathcal I_0e^{\varepsilon/(2d)}.
\end{equation*}
\end{theorem}
\begin{proof}
Write $F_{\mathrm{det}}={D_H}F$.
Lemma~\ref{lem:marked-rank-one} and Jensen give
\[
 \partial_{x_i}F_{\mathrm{det}}
 \ge {D_\perp}\{b_i^{\mathrm{ent}}-\log(2-Q/{D_H})\}
 \ge {D_\perp}\{\underline b_i^{\mathrm{ent}}-\log(2-\underline q_i)\}.
\]
Since $\underline q_i\le Q/D_H\le1$, both logarithms are defined, and $\log(2-q)\le\log2-q/2$ gives the linear test.
Thus the maximal-coordinate derivatives satisfy the asserted bounds.

Cap $x_i(r)=\min\{x_i,r\}$ for $0<r\le\max_i x_i$.
Between its finitely many break points,
\[
 \frac{d}{dr}F_{\mathrm{det}}(t,x(r))
 =\sum_{i:x_i>r}\partial_{x_i}F_{\mathrm{det}}(t,x(r))\ge0,
\]
because every active coordinate is maximal.
For $t>0$, $F_{\mathrm{det}}(t,0)=0$.
For $t=0$, fixed positive support keeps $H>0$ along the path, which near zero is $x(r)=r\1_{\{x>0\}}$.
Homogeneity and $\sum_e w_eh_e=dD_H$ make $F$ scale invariant and $F_{\mathrm{det}}$ homogeneous of degree $d$, again with zero initial limit.
Integration proves nonnegativity and strictness.
Applying the same argument to $F_{\mathrm{det}}-\varepsilon D_H$ proves the boundary gap.
Theorem~\ref{lem:balanced-entropy-bridge} and its scalar-core equality remark give the surface assertions.
\end{proof}

For the configurations studied here, the facet formulas below verify the algebraic hypotheses, and Sections~\ref{sec:global-comparison} and~\ref{sec:mean-bounds} establish the transverse entropy and cofactor-mean bounds.
Application to another family requires identifying its actual facet weights and squared length ratios and verifying all the hypotheses of the criterion.
Cofactor balance alone does not supply the two geometric bounds.
Scale invariance allows these bounds to be checked on $t+\max_i x_i=1$.
A larger geometric class also requires a covering by the relevant cones and separate comparisons for any remaining families.
Geometric stability additionally requires metric coercivity and control of how the cells vary with the metric.
The exact remainder supplies affine control relative to the facet matrix.

\subsection{Nonprincipal geometry and the marked facet measure}
\label{sec:nonprincipal}

The two nonprincipal families share a $D_4$ core and eight of their nine retained rays.
We also use the cographic configuration $\CoGr(G)$, whose bases are complements of spanning trees of a connected graph $G$.
For $v\ne0$, write $\Pi_v=vv^{\mathsf T}/\|v\|^2$, and for an index set $J$ write $x_J=\prod_{i\in J}x_i$, with $x_\varnothing=1$.
The retained rays in $T_j=\{3j-2,3j-1,3j\}$ and the omitted vector $m_j$ form the orthonormal frame $\mathrm{fr}_j$.
The type II coordinates are
\[
  \begin{array}{c|cccc}
  j&a_{3j-2}&a_{3j-1}&a_{3j}&m_j\\
  \hline
  1&e_1&e_2&e_3&e_4\\
  2&\tfrac12(1,1,1,-1)&\tfrac12(1,1,-1,1)&\tfrac12(1,-1,1,1)&\tfrac12(1,-1,-1,-1)\\
  3&\tfrac12(1,1,1,1)&\tfrac12(1,1,-1,-1)&\tfrac12(1,-1,-1,1)&\tfrac12(1,-1,1,-1).
  \end{array}
\]
Type III interchanges $a_9$ and $m_3$.
The twelve lines in the complete frames will be called the full-frame lines, with weight zero assigned to each omitted vector.
Set
\[
  B_{D_4}=\begin{pmatrix}1&0&0&0\\-1&1&0&0\\0&-1&1&1\\0&0&-1&1\end{pmatrix},
  \quad
  K=\sum_{i=1}^9x_i\Pi_{a_i},
  \quad
  H=t\Id+K
  ,\quad
  H_{\mathbb Z}=B_{D_4}^{\mathsf T}HB_{D_4},
\]
where $t,x_i\ge0$ and $H>0$.
Then $\det B_{D_4}=2$.
The fixed matrix $B_{D_4}$ records the lattice basis, while $H$ expresses the variable form in the orthonormal coordinates of the regular cell.
We call $t$ the \emph{core weight} and $x_1,\ldots,x_9$ the \emph{retained weights}.
When all retained weights vanish, $t$ only scales the regular cell.
Write $D_H=\det H$, $h_e=e^{\mathsf T}He$, $s_j=\sum_{i\in T_j}x_i$, ${s_{\rm tot}}=\sum_js_j$, and $u_j={s_{\rm tot}}-s_j$.
The old normal lines are
\[
  \mathcal N_{\rm old}=\{(e_r\pm e_s)/\sqrt2:1\le r<s\le4\}
  \quad\text{and}\quad
  \mathcal N=\mathcal N_{\rm old}\cup\{m_1,m_2,m_3\}.
\]
Each element of $\mathcal N$ represents an opposite facet pair, and each area below refers to one facet.

\begin{lemma}
\label{lem:frame-geometry}
The displayed rays represent the two closed nonprincipal $L$-domains, with retained configurations $\Gr(K_5-e)$ in type II and $\CoGr(K_{3,3})$ in type III.
Distinct frames have squared inner products $1/4$, and
\[
  \sum_{e\in\mathcal N_{\rm old}}\Pi_e=3\Id
  \quad\text{and}\quad
  \sum_{e\in\mathcal N_{\rm old},\, e\perp a}\Pi_e=2(\Id-\Pi_a)
\]
for every full-frame vector $a$.
For every symmetric $4\times4$ matrix $Z$,
\begin{equation}\label{eq:D4-root-moments}
 \sum_{e\in\mathcal N_{\rm old}}(e^{\mathsf T}Ze)\Pi_e
 =Z+\frac12\tr(Z)\Id\quad\text{and}\quad
 \sum_{e\in\mathcal N_{\rm old}}\operatorname{var}_e(Z)=2\|\tf Z\|^2.
\end{equation}
An independent triple of full-frame lines is either a same-frame triple of Gram determinant one or has Gram determinant $1/2$ and an old perpendicular normal.
Every independent retained four-tuple has squared determinant $1/4$.
For each old normal, the six perpendicular full-frame lines form the $A_3$ root configuration.
They have $16$ independent triples, each line belongs to eight, and an orthogonal or nonorthogonal pair extends to four or three, respectively.
The retained-ray symmetry orbits are $\{1,2,4,5,7,8\}$ and $\{3,6,9\}$ in type II, and all nine rays in type III.
\end{lemma}
\begin{proof}
The cone and graph identifications are the explicit $D_4$-ray description in \cite{DutourGrishukhin2009}.
The coordinates give the frame inner products and triple determinants.
An independent fourth retained ray has squared height $1/4$ over a same-frame triple and $1/2$ otherwise.
Both give squared determinant $1/4$.
Summing signed coordinate-pair projectors gives the first resolution, and frame symmetry gives the second.
Weighting these pairs by $e^{\mathsf T}Ze$ gives the matrix moment identity, and pairing it with $Z$ gives the variance identity.
The six lines perpendicular to an old normal have coordinates $(b_i\pm b_k)/\sqrt2$ in an orthonormal basis of its perpendicular space, so the remaining incidence counts are the spanning-tree counts of $K_4$.
The type II permutations $(1\,2)(4\,5)(7\,8)$, $(1\,4)(2\,5)(3\,6)$, and $(4\,8)(5\,7)(6\,9)$ are induced by orthogonal maps and give the stated orbits.
In type III, row permutations and simultaneous column permutations of the $3\times3$ ray array act orthogonally and transitively on the retained lines.
\end{proof}

We compute the facets in the dual cell, where the parameters act by Minkowski addition in fixed directions, and then apply $H^{-1/2}$.
The cubic weights below are both the facet-volume coefficients and the coefficients in an adjugate resolution.

\begin{lemma}
\label{lem:facet-cofactors}
Up to an orthogonal map, the Voronoi cell is
\[
  V_H=H^{-1/2}\mathcal C_H
  \quad\text{and}\quad
  \mathcal C_H=tC_{24}+\sum_{i=1}^9\frac{x_i}{2}[-a_i,a_i].
\]
For $e\in\mathcal N_{\rm old}$, write $I_e=\{i:e\perp a_i\}$ and define
\begin{equation*}
w_e=\frac{t^3}{3}+\frac{t^2}{2}\sum_{i\in I_e}x_i
+\frac t2\sum_{\substack{i<k\\i,k\in I_e}}(1+\1_{a_i\perp a_k})x_ix_k
+\frac12\sum_{\substack{J\subset I_e,\ |J|=3\\J\ \mathrm{independent}}}x_J.
\end{equation*}
For a new normal set $w_{m_j}={x_{T_j}}$, and note that $h_{m_j}=t+u_j/4$.
Then
\[
  \sum_{e\in\mathcal N}w_e\Pi_e=\operatorname{adj}H
  \quad\text{and}\quad
  \sum_{e\in\mathcal N}w_eh_e=4D_H.
\]
Every positive-area facet occurs in this list, with area
\[
  A_e=\frac{\sqrt2\,w_e\sqrt{h_e}}{\sqrt{D_H}}\quad\text{with } e\in\mathcal N_{\rm old},
  \quad\text{and}\quad
  A_{m_j}=\frac{{x_{T_j}}\sqrt{h_{m_j}}}{\sqrt{D_H}},
\]
and $\vol(V_H)=2\sqrt{D_H}$.
\end{lemma}
\begin{proof}
The vectors $B_{D_4}^{\mathsf T}a_i$ are primitive integral, and the isotropic dual cell is $2B_{D_4}^{\mathsf T}C_{24}$.
Lemma~\ref{lem:minkowski} therefore gives $\operatorname{Vor}_{\rm dual}(H_{\mathbb Z})=2B_{D_4}^{\mathsf T}\mathcal C_H$.
The orthogonal matrix $H_{\mathbb Z}^{-1/2}B_{D_4}^{\mathsf T}H^{1/2}$ identifies $\tfrac12H_{\mathbb Z}^{-1/2}\operatorname{Vor}_{\rm dual}(H_{\mathbb Z})$ with $V_H$, of volume $\sqrt{\det H_{\mathbb Z}}=2\sqrt{D_H}$.
An old normal exposes the octahedron with centered vertices $\pm b_k/\sqrt2$ and tangent directions $(b_i\pm b_k)/\sqrt2$.
Multilinearity gives area $\sqrt2w_e$.
Orthogonal and nonorthogonal pairs have coefficients two and one, respectively, while independent triples have Gram determinant $1/2$.
A new normal exposes a vertex of $C_{24}$ and its three orthogonal tangent segments, giving area ${x_{T_j}}$.
To verify cofactor balance, set $y_a=x_a+t/3$ on the twelve full-frame lines.
Since $H=\sum_a y_a\Pi_a$, adjugate Cauchy--Binet gives
\[
  \operatorname{adj}H=\sum_{e\in\mathcal N_{\rm old}}\widetilde w_e\Pi_e+\sum_{a\ \mathrm{full}}\widetilde w_a\Pi_a,
\]
where
\[
  \widetilde w_e=\frac12\sum_{\substack{|J|=3,\ J\perp e\\J\ \mathrm{independent}}}y_J \quad\text{and}\quad
  \widetilde w_a=\prod_{b\in \mathrm{fr}(a)\setminus\{a\}}y_b.
\]
Set $w_a=\prod_{b\in \mathrm{fr}(a)\setminus\{a\}}x_b$, which vanishes unless $a$ is omitted.
Incidence counts above give
\begin{align*}
w_e-\widetilde w_e&=\frac{t^3}{27}+\frac{t^2}{18}\sum_{a\perp e}x_a
+\frac t3\sum_{\substack{a<b,\ a,b\perp e\\a\perp b}}x_ax_b,\\
\widetilde w_a-w_a&=\frac{t^3}{27}+\frac{t^2}{9}\sum_{b\in \mathrm{fr}(a)\setminus\{a\}}x_b
+\frac t3\sum_{\{b,c\}\subset \mathrm{fr}(a)\setminus\{a\}}x_bx_c.
\end{align*}
Their projector sums agree.
The constant and linear terms use the resolutions.
For an orthogonal pair, complementary frame vectors and perpendicular old lines give orthonormal bases of the same plane.
Tracing against $H$ gives the scalar identity.
The support distances in $\mathcal C_H$ are $h_e/\sqrt2$ on old normals and $h_{m_j}$ on new ones.
The area factor $\sqrt{h_e}/\sqrt{D_H}$ under $H^{-1/2}$ gives the stated formulas.
The listed pairs exhaust $\sum_Fh_FA_F=4\vol(V_H)$, where $h_F$ is the facet's support distance from the center.
Thus no further positive-area facet occurs.
Continuity includes all zero parameters with $H>0$.
\end{proof}

Define the normalized marked cofactor mass by
\[
  \theta=\frac1{D_H}\sum_{j=1}^3{x_{T_j}}h_{m_j}.
\]
The two pyramids over the facet pair represented by $e$ occupy the fraction $w_eh_e/(4D_H)$ of $V_H$, so new facets account for the volume fraction $\theta/4$.

Use the entropy and facet matrices of Theorem~\ref{lem:balanced-entropy-bridge}, with $\kappa_e=1$ on old normals and $\kappa_e=2$ on new normals.
These marks are the squared facet-vector length ratios prescribed by the surface formula.
Thus $F=E-\theta\log2$ and $j_\kappa=\det W_\kappa/D_H^{5/2}$, and the actual unit normals remain $n_e=H^{1/2}e/\sqrt{h_e}$.
\begin{corollary}\label{lem:surface-calibration}\label{lem:entropy-bridge}
For every $|\det L|=1$,
\[
 \mathcal I(LV_H)\ge2^{11/4}j_\kappa^{1/4}
 \ge2^{11/4}\exp(F/8).
\]
At $H=t\Id$, $t>0$, one has $F=0$ and $j_\kappa=1$.
Equality in the surface comparison holds exactly for orthogonal $L$.
\end{corollary}
\begin{proof}
Lemma~\ref{lem:facet-cofactors} identifies the actual facet data with $d=4$, $v_\Lambda=2$, and $\ell_0=\sqrt2$, so $\mathcal I_0=2^{11/4}$.
Apply Theorem~\ref{lem:balanced-entropy-bridge} with marked sum $\sum_e\lambda_e\log\kappa_e=\theta\log2$.
Its scalar-core equality applies because only old weights are then positive.
\end{proof}

It remains to prove $F\ge0$, with equality only at the regular core.
This is the content of Theorem~\ref{thm:maximal-weight-entropy}.

\section{Transverse estimates and entropy descent}
\label{sec:global-comparison}

The main task is to bound transverse entropy strongly enough to offset the facet-length penalty.
The spanning-tree and variance estimates below, combined with the mean comparison of Section~\ref{sec:mean-bounds}, give positivity in the rank-one identity. Integration then proves the global comparison.

\subsection{Transverse entropy and variance}
\label{sec:transverse}

Choose a retained ray $a$ of weight $R$, with companion weights $u,v$, and write $H=R\Pi_a+{H_-}$.
Set ${H_\perp}={H_-}|_{a^\perp}$, ${D_\perp}=\det {H_\perp}$, and $h=t+s/4$, where $s$ is the total outside its frame.
If $w_e=R\omega_e+\zeta_e$ for $e\perp a$, define
\[
 E_\perp=\sum_{e\perp a}\omega_e\operatorname{ent}_e({H_\perp}).
\]
We bound $E_\perp+(D_\perp-uvh)\log2$ through the spanning-tree law of the six projected generators of $K_4$, and $E_\perp$ itself through their variance.
These complementary estimates apply in the two core ranges of Theorem~\ref{thm:maximal-weight-entropy}.

A pair shadow deletes one edge uniformly from a three-edge set, and its singleton shadow deletes one further edge.
Write $\operatorname{Sh}(p)=-\sum p\log p$ for Shannon entropy, with natural logarithms and $0\log0=0$.

The tree references are projection determinantal measures~\cite{Lyons2003}.
Cesaroni--Novaga~\cite[Sections~5.1 and~6.2]{CesaroniNovaga2026} contract entropy between product-weighted laws through singleton marginals.
Here facet cofactors require the pair shadow and an arbitrary comparison law.
The required contraction is a case of entropic independence~\cite{AnariEtAl2021}; we give a direct rank-two proof.
The sharp constant and equality case require the additional constraint imposed by product weights and the optimization below.

\begin{proposition}
\label{prop:sharp-K4-product-entropy}
For nonnegative edge weights $y_e$ on $K_4$ with $Z=\sum_{T\text{ tree}}y_T>0$, let $\mu_3(T)=y_T/Z$ and let $\mu_2$ be its uniform pair shadow.
Then
\[
  3\operatorname{Sh}(\mu_2)-2\operatorname{Sh}(\mu_3)+2\log2-3\log3
  \ge\frac94\log\frac43.
\]
Equality holds exactly when the six edge weights are equal up to a common scale.
\end{proposition}
\begin{proof}
For any tree law $\mu_3$ on $K_4$ with pair shadow $\mu_2$, let $p_i^{\mathrm{opp}}$ be the probability of containing opposite pair $i$, and set $p_{\mathrm{opp}}=\sum_i p_i^{\mathrm{opp}}$.
Every tree is a star or contains exactly one opposite pair. We first prove
\[
  \begin{aligned}
  E_{\mathrm{del}}(\mu_3)&:=3\operatorname{Sh}(\mu_2)-2\operatorname{Sh}(\mu_3)+2\log2-3\log3\\
  &\ge\Phi(\mathbf p^{\mathrm{opp}})
   :=(3-p_{\mathrm{opp}})\log\frac4{3-p_{\mathrm{opp}}}
  +\sum_i(1-p_i^{\mathrm{opp}})\log(1-p_i^{\mathrm{opp}}).
  \end{aligned}
\]
If a pair law $\nu$ and its singleton shadow $\pi$ satisfy $\sum_{i<j}\nu_{ij}z_iz_j\le(\sum_i\pi_i z_i)^2$ for $z_i>0$, testing the entropy variational formula with $\log z_i+\log z_j$ gives the rank-two principle
\[
 \operatorname{KL}(\mu\|\nu)
 \ge2\sup_{z>0}\left\{\mathbb E_{\mu_1}\log z-
                    \log\sum_i\pi_i z_i\right\}
 =2\operatorname{KL}(\mu_1\|\pi).
\]
Here $\mu$ is any pair law and $\mu_1$ its singleton shadow.
For a positive product-weighted tree reference $\nu_3$ with weights $\rho_i$, write $\mu_j,\nu_j$ for the shadows and $\operatorname{KL}_j=\operatorname{KL}(\mu_j\|\nu_j)$.
Using reduced incidence vectors $r_i$, define
\[
 \Sigma_\rho=\sum_i\rho_ir_ir_i^{\mathsf T},\quad\;
 \xi_i=\sqrt{\rho_i}\Sigma_\rho^{-1/2}r_i,\quad\;
 \mathsf C(z)=\sum_i z_i\xi_i\xi_i^{\mathsf T}.
\]
Cauchy--Binet, $\mathsf C(\1)=\Id_3$, and the tree determinants $\pm1$ give the reference polynomial $\det\mathsf C(z)$.
The linear coefficient of $\det(\mathsf C(z)+\varepsilon\Id_3)$ gives pair shadow $e_2(\mathsf C(z))/3$, with singleton mean $\tr\mathsf C(z)/3$.
Thus $e_2(\mathsf C)\le(\tr\mathsf C)^2/3$ yields $\operatorname{KL}_2\ge2\operatorname{KL}_1$.

Conditioning on a uniform tree edge $e$, whose reference probability is $\nu_1(e)=\|\xi_e\|^2/3$, leaves the pair law
\[
 \nu_{2\mid e}(\{i,j\})
 =\frac{\nu_3(\{e,i,j\})}{\|\xi_e\|^2}
 =\det_{\xi_e^\perp}(\Pi_{\xi_e^\perp}\xi_i,\Pi_{\xi_e^\perp}\xi_j)^2,
\]
where $\Pi_{\xi_e^\perp}$ projects onto $\xi_e^\perp$.
The identity $\sum_{i\ne e}(\Pi_{\xi_e^\perp}\xi_i)(\Pi_{\xi_e^\perp}\xi_i)^{\mathsf T}=\Id_2$ and the two-dimensional determinant inequality give the rank-two principle conditionally.
Distinguishing a uniform edge preserves relative entropy, so the chain rule, summing only over $\mu_1(e)>0$, gives
\[
 \operatorname{KL}_3-\operatorname{KL}_1
 =\sum_e\mu_1(e)\operatorname{KL}(\mu_{2\mid e}\|\nu_{2\mid e})
 \ge2\sum_e\mu_1(e)\operatorname{KL}(\mu_{1\mid e}\|\nu_{1\mid e})
 =2(\operatorname{KL}_2-\operatorname{KL}_1).
\]
Selecting a second uniform edge gives the same law as distinguishing an edge of the pair shadow, proving the last equality.
Combining the bounds yields $2\operatorname{KL}_3\ge3\operatorname{KL}_2$.

Give opposite pair $i$ equal reference weights $k_i>0$, with $k_{\mathrm{tot}}=\sum_i k_i$ and tree partition $Z_k=2(k_1+k_2)(k_2+k_3)(k_3+k_1)$.
The completing weights are $2(k_{\mathrm{tot}}-k_i)$ for opposite pair $i$ and $k_{\mathrm{tot}}$ for adjacent pairs, whose respective masses under $3\mu_2$ are $p_i^{\mathrm{opp}}$ and $3-p_{\mathrm{opp}}$ in total.
Expanding $2\operatorname{KL}_3-3\operatorname{KL}_2\ge0$ cancels edge logarithms by the $2/3$ retention probability and gives
\[
 {E_{\mathrm{del}}}(\mu_3)\ge\log(4Z_k)
 -\sum_i{p_i^{\mathrm{opp}}}\log[2(k_{\mathrm{tot}}-k_i)]
 -(3-p_{\mathrm{opp}})\log k_{\mathrm{tot}}.
\]
Taking $k_i/k_{\mathrm{tot}}=(1-p_{\mathrm{opp}}+2{p_i^{\mathrm{opp}}})/(3-p_{\mathrm{opp}})$ proves the bound by $\Phi$. Zero reference weights follow by a limit.

For the product-weighted law, the star probability $p_{\mathrm{star}}=1-p_{\mathrm{opp}}$ satisfies
\[
  {p_{\mathrm{star}}}^2\ge4{p_1^{\mathrm{opp}}}{p_2^{\mathrm{opp}}}{p_3^{\mathrm{opp}}}.
\]
First assume positive weights. Let ${r_i^{\mathrm{opp}}}$ be the geometric mean in opposite pair $i$ and $c_i\ge1$ the ratio of its arithmetic and geometric means.
Write the total star weight as $Z_{\mathrm{star}}=4{r_1^{\mathrm{opp}}}{r_2^{\mathrm{opp}}}{r_3^{\mathrm{opp}}}c_{\mathrm{st}}$.
Grouping the stars according to their edge in pair $i$ and applying AM--GM to each remaining two-term sum gives $c_{\mathrm{st}}\ge c_i$ for every $i$.
The tree weight and shape probabilities are
\[
  Z=2\sum_{i\ne j}{r_i^{\mathrm{opp}}}^2{r_j^{\mathrm{opp}}}c_j
      +4{r_1^{\mathrm{opp}}}{r_2^{\mathrm{opp}}}{r_3^{\mathrm{opp}}}c_{\mathrm{st}},
\]
\[
  p_i^{\mathrm{opp}}=\frac{2{r_i^{\mathrm{opp}}}^2
       ({r_j^{\mathrm{opp}}}c_j+{r_k^{\mathrm{opp}}}c_k)}{Z} \quad\text{and}\quad
  p_{\mathrm{star}}=\frac{Z_{\mathrm{star}}}{Z}.
\]
Consequently
\[
 \begin{aligned}
 c_{\mathrm{st}}^2Z-2\prod_{i<j}
          ({r_i^{\mathrm{opp}}}c_i+{r_j^{\mathrm{opp}}}c_j)
 &= 2\sum_{i\ne j}{r_i^{\mathrm{opp}}}^2{r_j^{\mathrm{opp}}}c_j
                       (c_{\mathrm{st}}^2-c_i^2)
 +4{r_1^{\mathrm{opp}}}{r_2^{\mathrm{opp}}}{r_3^{\mathrm{opp}}}
                       (c_{\mathrm{st}}^3-c_1c_2c_3) \\
                       &\ge0.
 \end{aligned}
\]
Dividing by $Z^3/(16{r_1^{\mathrm{opp}}}^2{r_2^{\mathrm{opp}}}^2{r_3^{\mathrm{opp}}}^2)$ proves compatibility.
Equality forces $c_{\mathrm{st}}=c_i$ for all $i$. The two AM--GM equalities for each pair force internal balance in the other pairs.
Continuity includes zero weights.

If ${p_{\mathrm{star}}}\ge1/4$, convexity gives $\Phi\ge(2+{p_{\mathrm{star}}})\log(4/3)\ge(9/4)\log(4/3)$.
For $0\le{p_{\mathrm{star}}}<1/4$, minimize $\Phi$ at fixed ${p_{\mathrm{star}}}$ subject to the product constraint.
On the coordinate boundary, equalizing the other two coordinates gives
\[
  \Phi\ge(2+{p_{\mathrm{star}}})\log4+(1+{p_{\mathrm{star}}})\log\frac{1+{p_{\mathrm{star}}}}{2}
  -(2+{p_{\mathrm{star}}})\log(2+{p_{\mathrm{star}}})\ge\log2>\frac94\log\frac43.
\]
The middle expression has derivative $\log[2(1+{p_{\mathrm{star}}})/(2+{p_{\mathrm{star}}})]\ge0$.
At an interior minimum the product constraint is active, since otherwise the infeasible all-equal vector is forced.
The multiplier equation $x(\log(1-x)-\lambda)=\gamma$ and strict concavity permit at most two coordinate values.
After relabeling,
\[
  {p_1^{\mathrm{opp}}}={p_2^{\mathrm{opp}}}=\frac1{2(1+z)},
  \quad{p_3^{\mathrm{opp}}}=\frac{z^2}{(1+z)^2},
  \quad{p_{\mathrm{star}}}=\frac{z}{(1+z)^2},\quad\text{with } z>0.
\]
The resulting value of $\Phi$ is
\[
  \varphi(z)=\frac{1+2z}{(1+z)^2}
  \bigl[(z+1)\log(z+1)-(z+2)\log(z+2)+(z+3)\log2\bigr].
\]
For $\psi(z)=(1+z)^3\varphi'(z)$ one has
\[
  \begin{aligned}
  \psi(z)&=(z+1)\log(z+1)+(z-1)\log(z+2)+(1-3z)\log2,
  \\
  \psi''(z)&=\frac1{z+1}+\frac1{z+2}+\frac3{(z+2)^2}>0.
  \end{aligned}
\]
Since $\psi(0)=\psi(1)=0$, the unique minimum is $\varphi(1)=(9/4)\log(4/3)$.
Equality requires $p_{\mathrm{star}}=p_i^{\mathrm{opp}}=1/4$, forcing all six weights positive.
Compatibility equality balances each opposite pair, and the equal shape probabilities then equate the three pair weights.
\end{proof}

The transverse calibration uses the following basis identity.

\begin{proposition}
\label{prop:basis-entropy-identity}
Let $H=\sum_i x_iv_iv_i^{\mathsf T}>0$, where $v_i\in\R^d$, $x_i\ge0$, and every nonzero squared basis determinant equals $\gamma>0$.
The product-weighted basis law and its uniform deletion shadow are
\[
 \mu_d(B)=\frac{\gamma x_B}{{D_H}}\quad\text{and}\quad
 \mu_{d-1}(J)=\frac1d\sum_{B\supset J}\mu_d(B).
\]
For each independent $(d-1)$-set $J$, let $g_J$ be its Gram determinant and $n_J$ a unit normal.
The indexed cofactor weights $w_J=g_Jx_J$ (with repeated normal lines allowed) and marks $\kappa_J=cg_J$, $c>0$, give the compensated entropy
\[
 F_\kappa=d\operatorname{Sh}(\mu_{d-1})-(d-1)\operatorname{Sh}(\mu_d)
 -d\log d-(d-1)\log\gamma-d\log c.
\]
\end{proposition}
\begin{proof}
For positive weights let $Z=\sum_{B\text{ basis}}x_B$ and $s_J=\sum_{i:J\cup\{i\}\text{ basis}}x_i$.
Cauchy--Binet gives $D_H=\gamma Z$ and $\sum_Jw_J\Pi_{n_J}=\operatorname{adj}H$.
Adjoining one vector to $J$ gives
\[
 h_J=n_J^{\mathsf T}Hn_J=\frac{\gamma s_J}{g_J}\quad\text{and}\quad
 \lambda_J=\frac{w_Jh_J}{{D_H}}=d\mu_{d-1}(J).
\]
Uniform deletion gives $d\Pr_{\mu_{d-1}}(i\in J)=(d-1)\Pr_{\mu_d}(i\in B)$, so the individual generator logarithms cancel in
\[
 d\operatorname{Sh}(\mu_{d-1})-(d-1)\operatorname{Sh}(\mu_d)
 =d\log d+\log Z-d\sum_J\mu_{d-1}(J)\log s_J.
\]
Substituting $\kappa_Jh_J=c\gamma s_J$ and $D_H=\gamma Z$ into $F_\kappa=\log D_H-\sum_J\lambda_J\log(\kappa_Jh_J)$ proves the identity.
Continuity includes zero weights whenever $D_H>0$.
\end{proof}

The next proposition identifies the transverse entropy with the quantity of Proposition~\ref{prop:sharp-K4-product-entropy}.
Cauchy--Binet assigns every projected basis the product of its generator weights, so the transverse law is product-weighted.
The companion pair has Gram determinant one instead of $1/2$, and this exceptional factor produces the marked term $-uvh\log2$.

\begin{proposition}
\label{prop:transverse-K4-calibration}
For the transverse metric $H_\perp$, whose determinant is $D_\perp$, use the two companion rays and the missing frame direction as coordinates.
Then
\[
  {H_\perp}:={H_-}|_{a^\perp}=\operatorname{diag}(u,v,0)+W_{\mathrm{tet}},
  \quad\text{where }\; W_{\mathrm{tet}}=\sum_{i=1}^4z_ib_i^{\perp}\bigl(b_i^{\perp}\bigr)^{\mathsf T}\;\text{ with }\; z_i\ge0,
\]
where
\[
  b_1^{\perp}=\frac12(1,1,1),\quad b_2^{\perp}=\frac12(1,-1,-1),
  \quad b_3^{\perp}=\frac12(-1,1,-1),
  \quad b_4^{\perp}=\frac12(-1,-1,1).
\]
The diagonal of $W_{\mathrm{tet}}$ is $(h,h,h)$, and
\[
  \sum_{e\perp a}\omega_e\Pi_e=\operatorname{adj}{H_\perp}
  \quad\text{and}\quad
  \frac{E_\perp+({D_\perp}-uvh)\log2}{{D_\perp}}
  \ge\frac94\log\frac43.
\]
The constant is sharp.
\end{proposition}
\begin{proof}
Outside rays project onto the four lines $b_i^\perp$. Since $\sum_i b_i^\perp(b_i^\perp)^{\mathsf T}=\Id_3$, the core adds $t$ to each $z_i$, giving diagonal $h=t+s/4$.
The coefficient of $R$ in Lemma~\ref{lem:facet-cofactors} gives the transverse adjugate resolution: the missing frame normal is the third coordinate with cofactor $uv$, and the six old normals are $(e_i\pm e_j)/\sqrt2$.

The two coordinate generators and $b_1^{\perp},b_2^{\perp},b_3^{\perp},b_4^{\perp}$ represent the edges $12,34,13,23,14,24$ of $K_4$, respectively.
Indeed, their four triangle relations follow from
\[
  b_1^{\perp}+b_2^{\perp}=-b_3^{\perp}-b_4^{\perp}=(1,0,0)
  \quad\text{and}\quad
  b_1^{\perp}+b_3^{\perp}=-b_2^{\perp}-b_4^{\perp}=(0,1,0).
\]
Tree triples have squared determinant $1/4$, and pairs have Gram determinant $1/2$ except the opposite coordinate pair, whose determinant is one.
Let $y=(u,v,z_1,z_2,z_3,z_4)$, set $Z=\sum_{T\text{ tree}}y_T=4{D_\perp}$, and let $\mu_3(T)=y_T/Z$.

The pair-cofactor expansion equals $\omega$. Both assign $uv$ to the marked normal, and balance determines the other weights because the six old projectors form a basis of the symmetric $3\times3$ matrices.
For the pair Gram determinants $g_J$, the marks $\kappa_J=2g_J$ are one except at the coordinate pair, of cofactor mass $uvh/D_\perp$.
Proposition~\ref{prop:basis-entropy-identity} with $d=3$, $\gamma=1/4$, $c=2$, and pair shadow $\mu_2$ gives
\[
  \frac{E_\perp+({D_\perp}-uvh)\log2}{{D_\perp}}
  =3\operatorname{Sh}(\mu_2)-2\operatorname{Sh}(\mu_3)+2\log2-3\log3.
\]
Proposition~\ref{prop:sharp-K4-product-entropy} proves the bound, including zero weights by continuity. Equal positive weights attain equality.
\end{proof}

In the range $t\ge2R/3$ of Theorem~\ref{thm:maximal-weight-entropy} we need a bound on $E_\perp$ itself in units of $uvh$.
It follows from a variance bound and the curvature of $x\log x$.

\begin{proposition}
\label{prop:transverse-marked-variance}
For any $u,v,z_i\ge0$ in the preceding representation of $H_\perp$, use its pair-cofactor weights $\omega_e$, including the marked weight $uv$.
With $h$ the common diagonal of $W_{\mathrm{tet}}$,
\[
 V_\perp:=\sum_e\omega_e\operatorname{var}_e(H_\perp)
 \ge\frac{uvh}{2}\max_i\sum_j|(H_\perp)_{ij}|
 \quad\text{and}\quad E_\perp\ge\frac{uvh}{4}.
\]
Both constants are sharp.
\end{proposition}
\begin{proof}
Nonnegativity handles $uvh=0$. Otherwise, normalize $h=1$ and write
\[
  W_{\mathrm{tet}}=\begin{pmatrix}1&r_{12}&r_{13}\\r_{12}&1&r_{23}\\r_{13}&r_{23}&1\end{pmatrix},
  \quad {c_{20}}=1-r_{23}^2,
  \quad {c_{02}}=1-r_{13}^2,
\]
\[
  \begin{aligned}
  {c_{11}}&=-r_{12}^2+3r_{13}^2+3r_{23}^2-1,\\
  {c_{10}}&=2(r_{12}^2+r_{13}^2-2r_{12}r_{13}r_{23}),\\
  {c_{01}}&=2(r_{12}^2+r_{23}^2-2r_{12}r_{13}r_{23}).
  \end{aligned}
\]
Here $|r_{12}|,|r_{13}|,|r_{23}|\le1$ and ${c_{20}},{c_{02}},{c_{10}},{c_{01}}\ge0$.
The old normals $(e_i\pm e_j)/\sqrt2$ have weights ${w_{ij}^{\mathrm{sum}}}/2\pm\operatorname{adj}({H_\perp})_{ij}$, where
\[
  \begin{aligned}
  {w_{12}^{\mathrm{sum}}}&=1+r_{12}^2-r_{13}^2-r_{23}^2,\\
  {w_{13}^{\mathrm{sum}}}&=2v+1-r_{12}^2+r_{13}^2-r_{23}^2,\\
  {w_{23}^{\mathrm{sum}}}&=2u+1-r_{12}^2-r_{13}^2+r_{23}^2.
  \end{aligned}
\]
Adding their variances and the marked contribution gives
\[
  2V_\perp=uv(u+v)+{c_{20}} u^2+{c_{02}} v^2+uv{c_{11}}+u{c_{10}}+v{c_{01}}+{c_{00}},
\]
where
\[ {c_{00}}=\frac18\sum_{i<j}z_iz_j(z_k-z_l)^2\ge0.
\]
Here $\{k,l\}$ complements $\{i,j\}$.
The formula for $c_{00}$ uses weight $z_iz_j/2$ and variance $(z_k-z_l)^2/8$ at the normal to $b_i^{\perp},b_j^{\perp}$.

For the first-row bound, use coordinate sign changes to make $r_{12},r_{13}\ge0$. These preserve $r_{12}r_{13}r_{23}$ and the variance expansion.
Set $\alpha=\sqrt{1-r_{13}^2}$ and $\beta=\sqrt{1-r_{23}^2}$.
Cauchy--Schwarz gives $2\sqrt{{c_{10}}}\ge2(r_{12}-r_{13}r_{23}+r_{13}\beta)$.
Since $r_{12}-r_{12}^2\ge0$, $\alpha\ge1-r_{13}^2$, and $\beta\ge1-r_{23}^2$, it follows that
\[
  \begin{aligned}
  & {c_{11}}-1-r_{12}-r_{13}+2\sqrt{{c_{20}}{c_{02}}}+2\sqrt{{c_{10}}} \\
  \ge{}& {p_{\mathrm{aux}}} :=r_{13}^2+r_{13}-2r_{13}r_{23}+(1-2r_{13}+2r_{13}^2)r_{23}^2 \\ \ge{}& 0.
  \end{aligned}
\]
The last inequality follows from the identity
\[
  \chi {p_{\mathrm{aux}}}=(\chi r_{23}-r_{13})^2+r_{13}\bigl[(1-r_{13})^2(1-r_{13}^2)+r_{13}^4\bigr],\quad\text{where } \chi=1-2r_{13}+2r_{13}^2>0.
\]
AM--GM gives ${c_{20}} u^2+{c_{02}} v^2\ge2uv\sqrt{{c_{20}}{c_{02}}}$ and $uv^2+u{c_{10}}\ge2uv\sqrt{{c_{10}}}$.
The scalar inequality therefore yields
\[
  2V_\perp\ge uv(1+u+|r_{12}|+|r_{13}|).
\]
Interchanging the first two coordinates proves the second-row bound.
For the third row, additionally use $u^2v+v{c_{01}}\ge2uv\sqrt{{c_{01}}}$.
Its scalar remainder exceeds the first-row remainder by $|r_{12}|-|r_{23}|+2\sqrt{{c_{01}}}\ge0$, because ${c_{01}}\ge2(|r_{12}|-|r_{23}|)^2$.
Thus $2V_\perp\ge uv(1+|r_{13}|+|r_{23}|)$. Scaling restores $h$ and proves the rowwise estimate.

For $H_\perp>0$, spectral Jensen with $(x\log x)''\ge1/\lambda_{\max}(H_\perp)$ and the row-sum bound on $\lambda_{\max}(H_\perp)$ give $E_\perp\ge V_\perp/(2\lambda_{\max}(H_\perp))\ge uvh/4$.
Continuity includes singular matrices.
Finally, when $W_{\mathrm{tet}}=\Id$ and $u=v=\varepsilon\downarrow0$,
\[
  2V_\perp=\varepsilon^2+2\varepsilon^3
  \quad\text{and}\quad
  E_\perp=(1+2\varepsilon)\bigl[(1+\varepsilon)\log(1+\varepsilon)-(2+\varepsilon)\log(1+\varepsilon/2)\bigr]
  =\frac{\varepsilon^2}{4}+O(\varepsilon^3).
\]
This proves sharpness of both constants.
\end{proof}
\subsection{Entropy descent and completion}
\label{sec:descent}

We prove Theorem~\ref{thm:maximal-weight-entropy} by showing that $D_H F$ increases in every largest retained weight and integrating along the capped path $x_i(r)=\min\{x_i,r\}$.
The derivative consists of the transverse entropy of Subsection~\ref{sec:transverse} and a logarithmic correction, a cofactor-weighted mean over the normals not perpendicular to the chosen ray.
The following lemma, proved in Section~\ref{sec:mean-bounds}, bounds this correction through the common frame geometry and the maximality of $R$.

\begin{lemma}
\label{lem:maxray-mean-comparison}
In either nonprincipal configuration, choose a largest retained weight $R=x_i>0$ and write $H=R\Pi_{a_i}+H_->0$.
Thus $x_k\le R$ for every $k$.
Use the transverse quantities $H_\perp,D_\perp,u,v$ and $h=t+s/4$ above, and the mean numerator of Lemma~\ref{lem:marked-rank-one},
\[
 Q=RD_\perp-2\det H_-+\sum_{e\perp a_i}\zeta_eh_e.
\]
Then
\[
  \begin{aligned}
  9Q\ge D_H\qquad &\text{if}\quad  0\le t\le2R/3,\\
  Q\ge uvh(R+h)+\frac{23}{324}R^4\qquad &\text{if}\quad t\ge2R/3.
  \end{aligned}
\]
\end{lemma}

\begin{theorem}
\label{thm:maximal-weight-entropy}
In either nonprincipal type, every positive-definite form satisfies $F\ge0$, with equality exactly when all retained weights vanish.
Consequently $\mathcal I(LV_H)\ge2^{11/4}$ for every such form and every $|\det L|=1$, strictly whenever some retained weight is positive.
On the positive-definite zero-core boundary $t=0$, one has the uniform bounds
\begin{equation}\label{eq:zero-core-uniform-gap}
 F\ge\frac14\log\frac{139968}{130321}>0
 \quad\text{and}\quad
 \mathcal I(LV_H)\ge2^{11/4}
          \left(\frac{139968}{130321}\right)^{1/32}.
\end{equation}
\end{theorem}

\begin{proof}
For a largest weight $R=x_i>0$, set $y=Q/D_H$.
The explicit cofactors satisfy the hypotheses of Theorem~\ref{thm:portable-global-comparison}.
Old and new nonperpendicular normals have $(\alpha_e,\kappa_e)=(1/2,1)$ and $(1/4,2)$, so the angular constant is $b=\log2$.
The transverse mark has cofactor $uv$ and height $h$, so $E_{\mathrm{pen},\perp}=uvh\log2$ and
\[
 b_i^{\mathrm{ent}}=\frac{E_\perp+({D_\perp}-uvh)\log2}{{D_\perp}},
 \quad D_H={D_\perp}(R+{h_{\mathrm S}}),\quad\text{with}\quad 0\le{h_{\mathrm S}}\le h.
\]
If $t\le2R/3$, Proposition~\ref{prop:transverse-K4-calibration} and Lemma~\ref{lem:maxray-mean-comparison} give
\[
 b_i^{\mathrm{ent}}\ge\frac94\log\frac43,
 \qquad y\ge\frac19,
 \qquad
 \frac94\log\frac43-\log\frac{17}{9}
 =\frac14\log\frac{262144}{250563}>0.
\]
This verifies the criterion's strict logarithmic test.
If $t\ge2R/3$, Proposition~\ref{prop:transverse-marked-variance} and the second mean bound verify the strict linear test, since
\[
 \begin{split}
 {D_\perp}\left(b_i^{\mathrm{ent}}-\log2+\frac y2\right)
 &=E_\perp-uvh\log2+\frac{Q}{2(R+{h_{\mathrm S}})}\\
 &\ge uvh\left(\frac34-\log2\right)
       +\frac{23R^4}{648(R+{h_{\mathrm S}})} \\ &>0.
 \end{split}
\]
The criterion proves the comparison and equality assertions.

At $t=0$, Proposition~\ref{prop:sharp-zero-core-exchange} improves the exchange bound to $y\ge10/81$.
The logarithmic test then has the uniform margin
\[
 c_\partial=\frac94\log\frac43-\log\frac{152}{81}
 =\frac14\log\frac{139968}{130321}>0.
\]
The criterion's boundary clause gives the claimed uniform gap.
\end{proof}

\begin{proposition}
\label{prop:direct-metric-coercivity}
In either nonprincipal configuration, let $H=t\Id+K>0$, $K=\sum_{i=1}^9x_i\Pi_{a_i}$, and ${s_{\rm tot}}=\sum_i x_i$.
Set $c_0=(16\cdot35^4)^{-1}$.
For $t>0$ one has
\begin{equation*}
 F(H)\ge
 \begin{cases}
  s_{\rm tot}^2/(256t^2),&s_{\rm tot}\le t/16,\\
  c_0,&s_{\rm tot}>t/16.
 \end{cases}
\end{equation*}
Thus $F(H)\ge c_0\min\{(s_{\rm tot}/t)^2,1\}$. At $t=0$ one also has $F\ge c_0$.
Consequently, if $|\det L|=1$ and
\[
 \eta=\log\frac{\mathcal I(LV_H)}{2^{11/4}}<\frac{c_0}{8},
\]
then $t>0$, ${s_{\rm tot}}/t<1/16$, and
\begin{equation*}
 \frac{s_{\rm tot}} t\le\sqrt{2048\eta}\quad\text{and}\quad
 \|H/t-\Id\|_{\mathrm{op}}\le\sqrt{2048\eta}.
\end{equation*}
\end{proposition}
\begin{proof}
For $t>0$, scale to $t=1$ and set $M_{\rm omit}=\sum_j\Pi_{m_j}$.
The frame inner products give
\[
 \|\tf {M_{\rm omit}}\|^2=\frac94,\quad
 \langle\tf K,\tf {M_{\rm omit}}\rangle=-\frac{s_{\rm tot}}4,
 \quad {s_{\rm tot}}\le6\|\tf K\|.
\]
Spectral Jensen, \eqref{eq:D4-root-moments}, $w_e\ge1/3$ on old normals, and $\lambda_{\max}(H)\le1+s_{\rm tot}$ give
\[
 E\ge\frac{\|\tf K\|^2}{3D_H(1+{s_{\rm tot}})}
 \ge\frac{{s_{\rm tot}}^2}{108D_H(1+{s_{\rm tot}})}.
\]
Since $x_{T_j}\le s_j^3/27$ and $h_{m_j}\le1+s_{\rm tot}/4$, subtracting the marked penalty yields
\[
 F\ge\frac{{s_{\rm tot}}^2}{D_H}
 \left\{\frac1{108(1+{s_{\rm tot}})}
       -\frac{{s_{\rm tot}}(1+{s_{\rm tot}}/4)\log2}{27}\right\}.
\]
When ${s_{\rm tot}}\le1/16$, the first term in braces is at least $1/128$ and the second at most $1/384$.
Moreover $D_H\le(1+{s_{\rm tot}}/4)^4<4/3$.
This proves $F\ge{s_{\rm tot}}^2/256$ in the stated range.

For larger $s_{\rm tot}$, follow the capped path $x_i(r)=\min\{x_i,r\}$ of Theorem~\ref{thm:portable-global-comparison}.
Write $D_H(r)=\det H(r)$ and $F_{\mathrm{det}}(r)=D_H(r)F(H(r))$.
The latter is nondecreasing, and the small-core bound on every active coordinate gives, for $r\ge3/2$,
\[
 F_{\mathrm{det}}'(r)\ge c_*{D_H}'(r),\quad\text{where }
 c_*:=\frac14\log\frac{262144}{250563}>c_0,
\]
away from break points. The inequality $\log z\ge1-z^{-1}$ gives $c_*\ge11581/1048576>c_0$.
At and beyond the cap level $r_0$ with $\sum_i x_i(r_0)=1/16$,
\[
 F_{\mathrm{det}}(r)\ge F_{\mathrm{det}}(r_0)
 \ge {D_H}(r_0)/65536\ge1/65536.
\]
For $r\le3/2$, the trace bound gives
\[
 {D_H}(r)\le\left(1+\frac94r\right)^4
 \le(35/8)^4
 \quad\text{and}\quad F(H(r))\ge\frac1{65536(35/8)^4}=c_0.
\]
If the final cap exceeds $3/2$, integration of $F_{\mathrm{det}}'\ge c_*{D_H}'$ from that level gives
\[
 F_{\mathrm{det}}(r)\ge c_0{D_H}(3/2)
             +c_*\bigl({D_H}(r)-{D_H}(3/2)\bigr)\ge c_0{D_H}(r).
\]
This proves the entropy bound for $t>0$; the zero-core gap~\eqref{eq:zero-core-uniform-gap} covers $t=0$.
Finally, $F\le8\eta<c_0$ forces $t>0$ and $s_{\rm tot}/t<1/16$. The local bound and $\|K\|_{\mathrm{op}}\le\tr K=s_{\rm tot}$ give the metric estimates.
\end{proof}

\begin{corollary}\label{cor:entropy-geometric-control}
In the setting of Proposition~\ref{prop:direct-metric-coercivity}, write $L^{-\mathsf T}=Ue^Y$ with $Y=Y^{\mathsf T}$.
For the same logarithmic surface deficit $\eta<c_0/8$, one has $t>0$ and
\begin{equation*}
 \frac{s_{\rm tot}} t+\|Y\|_{\mathrm{op}}
 +d_{24}(LV_H)\le C\sqrt\eta
\end{equation*}
for an absolute constant $C$.
\end{corollary}
\begin{proof}
Proposition~\ref{prop:direct-metric-coercivity} gives $t>0$ and $q=s_{\rm tot}/t<1/16$, with $q\le\sqrt{2048\eta}$.
Since $t\Id\le H\le t(1+q)\Id$ and the attenuation factors are at most one, cofactor balance, $w_e\ge t^3/3$ on old normals, and $\sum_{\mathcal N_{\rm old}}\Pi_e=3\Id$ give
\[
 \frac{t^{5/2}}{\sqrt{1+q}}\Id\le W_\kappa
 \le t^{-1/2}\operatorname{adj}H.
\]
For $\widehat W_\kappa=H^{1/2}W_\kappa H^{1/2}$, congruence and AM--GM yield
\begin{align*}
 \frac{t^{7/2}}{\sqrt{1+q}}\Id\le\widehat W_\kappa
 &\le t^{7/2}(1+q/4)^4\Id,\\
 \|\tf\log\widehat W_\kappa\|_{\mathrm{op}}
 &\le4\log(1+q/4)+\frac12\log(1+q)\le\frac32q.
\end{align*}
Since $F\ge0$, Corollary~\ref{cor:exact-affine-coercivity} applies with remainder at most $\eta$ and gives $\|Y\|_{\mathrm{op}}\le C\sqrt\eta+3q/2\le C\sqrt\eta$.
Set $H_0=H/t$ and $D_{H_0}=\det H_0$.
The Minkowski formula gives
\[
 \widehat{LV_H}
 =U\,2^{-1/4}D_{H_0}^{-1/8}e^{-Y}H_0^{-1/2}
 \left(C_{24}+\frac12\sum_i\frac{x_i}{t}[-a_i,a_i]\right).
\]
The body in parentheses is within $q/2$ of $C_{24}$, and
\[
 \|H_0^{-1/2}-\Id\|_{\mathrm{op}}\le q/2,\qquad
 |D_{H_0}^{-1/8}-1|\le q/8,\qquad
 \|e^{-Y}-\Id\|_{\mathrm{op}}\le e^{\|Y\|_{\mathrm{op}}}-1.
\]
These factors are uniformly bounded, so comparison of support functions gives $d_{24}(LV_H)\le C(q+\|Y\|_{\mathrm{op}})$, completing the proof.
\end{proof}

\begin{proof}[Proof of Theorem~\ref{thm:main}]
By Theorem~\ref{thm:classification} and dilation invariance, it suffices to consider volume-preserving affine images in the three closed primitive cones.
Theorem~\ref{thm:zonotope-gap} is strict in the principal cone; Theorem~\ref{thm:maximal-weight-entropy} and Corollary~\ref{lem:entropy-bridge} are strict in the nonprincipal cones when $K\ne0$.
At $K=0$, Lemma~\ref{lem:reference-cell} and Corollary~\ref{lem:surface-calibration} give equality exactly for orthogonal images of the regular cell.
Restoring translations and dilations gives its similarities.
\end{proof}
\begin{proof}[Proof of Corollary~\ref{cor:fixed-inradius}]
Pass to the closure and choose $B(x_0,r)\subset P$.
The facet support distances satisfy $h_F\ge r$, so the pyramid formula gives $4\vol_4(P)\ge r\operatorname{area}(P)$, where $\operatorname{area}(P)=\mathcal H^3(\partial P)$.
Together with Theorem~\ref{thm:main},
\[
 \operatorname{area}(P)^4\ge2^{11}\vol_4(P)^3
 \ge32r^3\operatorname{area}(P)^3,
\]
proving the bound after division.
Equality forces a regular $24$-cell and $h_F=r$ for every facet, hence inradius $r$.
Conversely, Lemma~\ref{lem:reference-cell} and scaling give its area $32r^3$.
\end{proof}

\begin{remark}
\label{rem:24-cell-packing}
For a convex lattice fundamental domain containing a ball of radius $r$, Corollary~\ref{cor:fixed-inradius} and the pyramid formula give $\vol_4(P)\ge8r^4$, with equality only for the regular $24$-cell of inradius $r$.
For lattice Voronoi cells this recovers the optimal lattice packing density $\pi^2/16$, equivalently the value $\gamma_4=\sqrt2$ of the Hermite constant~\cite{ConwaySloane1999}, together with the uniqueness of $D_4$ up to similarity.
The $24$-cell conjecture~\cite{Musin2018} asks for the same cellwise volume bound for Voronoi cells of arbitrary packings of balls of radius $r$.
Since such cells need not be parallelohedra, Theorem~\ref{thm:main} does not settle that conjecture.
\end{remark}

\section{A unified frame comparison}
\label{sec:mean-bounds}

The two mean bounds in Lemma~\ref{lem:maxray-mean-comparison} come from a single geometric model: the six outside rays occupy two parity classes of a sign cube.
Rooted basis exchanges become its edges, diagonals, triangles, and stars.
This gives a common proof for both nonprincipal configurations, including the zero-core comparison and a positive large-core remainder.
The obstacle is to control negative exchange terms using the common frame geometry and the largest-weight constraint. Determinant expansion and interpolation organize the comparison.
Throughout, $R$ is a largest retained weight, so the other eight retained weights lie in $[0,R]$.
Its ray is $a$, its two companion weights are $u,v$, the outside total is $s$, and $h=t+s/4$.

The argument has three stages.
First, Lemmas~\ref{lem:linear-core-frame}--\ref{lem:cube-constant-bound} express the coefficients of $Q$ through cube incidences and bound their negative contributions.
These bounds give the large-core estimate in Lemma~\ref{lem:maxray-mean-packets} and the zero-core estimate in Proposition~\ref{prop:sharp-zero-core-exchange}.
Finally, the endpoint comparison at $t=2R/3$ and Lemma~\ref{lem:maxray-low-core} cover the intervening core range, completing Lemma~\ref{lem:maxray-mean-comparison}.

Recall the exchange numerator from the rank-one identity:
\[
 Q=RD_\perp-2\det H_-+\sum_{e\perp a}\zeta_eh_e.
\]
For a retained index set $J$, write
\[
  g_J=\det\bigl((a_i\cdot a_k)_{i,k\in J}\bigr),\quad\text{with}\quad g_\varnothing=1.
\]
Cauchy--Binet gives the coefficient formula
\begin{equation*}
 \det\left(t{\Id}+\sum_i x_i{\Pi_{a_i}}\right)
 =\sum_{J\subset\{1,\ldots,9\},\; |J|\le4}
       t^{4-|J|}g_Jx_J.
\end{equation*}
For ${D_\perp}=\det({H_-|_{a^\perp}})$, the analogous formula uses subsets avoiding the selected ray and the projected Gram entries $a_i\cdot a_k-(a_i\cdot a)(a_k\cdot a)$, with $t^{3-|J|}$ in place of $t^{4-|J|}$.
The remaining term in ${Q}$ is obtained by setting the selected weight to zero in the facet formula for $w_e$ and multiplying by $h_e$, for $e\perp a$.
Rooted basis exchanges determine the constant coefficient.
Weighted cube marginals control the linear coefficient, while elementary symmetric functions of the frame operator supply the dense-array bounds.

These formulas also show that ${Q}$ and $9{Q}-{D_H}$ are affine separately in every retained weight.
For the facet terms, a weight occurring in $w_e$ belongs to a ray perpendicular to $e$ and therefore does not occur in $h_e$.
We repeatedly use the following interpolation identity.
If $f$ is affine separately in $m$ variables and $0\le\xi_i\le1$, then
\begin{equation*}
 f(\xi)=\sum_{\varepsilon\in\{0,1\}^m}f(\varepsilon)
          \prod_{i:\varepsilon_i=1}\xi_i
          \prod_{i:\varepsilon_i=0}(1-\xi_i).
\end{equation*}
In particular, a lower bound at the vertices of the cube holds throughout the cube.
All expressions being compared are homogeneous of degree four in $(t,x)$.
When $R>0$, division by $R$ therefore reduces the cube $[0,R]^8$ to $[0,1]^8$.
These polynomial inequalities hold also where $H$ is singular.
Positive definiteness is used only in Subsection~\ref{sec:descent}.

Write ${Q}=q_0+tq_1+t^2q_2+t^3q_3$.
The fourth-degree core terms cancel because $-2{\det H_-}$ contributes $-2t^4$, and the six old normals perpendicular to $a$ contribute $2t^4$.
The common frame incidences give
\begin{equation}\label{eq:common-high-core-coefficients}
 q_3=R-\frac{u+v}{3}\quad\text{and}\quad
 q_2=R(u+v)-uv+\frac{3R-u-v}{4}s.
\end{equation}

\subsection{Frame identities and rooted exchanges}

The weight of a collection of vertex sets means the sum of the products of their vertex weights.

\begin{lemma}
\label{lem:linear-core-frame}
Use the selected complete frame as coordinates, with $a=e_0$, retained companions $e_1,e_2$ of weights $u,v$, and omitted ray $e_3$.
The outside rays, oriented toward $a$, are $b_\sigma=(1,\sigma_1,\sigma_2,\sigma_3)/2$. Their two frames are the two parity classes of the sign cube, with one vertex of weight zero in each class.
Write ${N_{{\rm edge},i}}$ for the weight of cube edges in direction $i$, ${N_{{\rm diag},i}}$ for the weight of same-parity pairs whose unchanged coordinate is $i$, and ${N_{\rm par}}$ for the weight of parity triples.
Then
\begin{align*}
 q_1={}&Ruv+\frac{R(u+v)-uv}{2}s\\
 &+\frac12\{R{N_{{\rm diag},3}}-{N_{\rm par}}+(R-u){N_{{\rm diag},1}}+(R-v){N_{{\rm diag},2}}\}\\
 &+\frac14\{(2R-v){N_{{\rm edge},1}}+(2R-u){N_{{\rm edge},2}}+(2R-u-v){N_{{\rm edge},3}}\}.
\end{align*}
In particular, if all retained weights are at most $R$, then $q_1\ge Ruv+uvs/2$.
\end{lemma}
\begin{proof}
The parity description follows because outside rays are orthogonal exactly when two signs differ.
The coefficient of $Rx_bx_c$ is the Gram determinant of $\{a,b,c\}$: one for two companions, $1/2$ for a companion--outside pair, a same-parity pair, or a cube edge, and zero for an antipodal pair.
An unrooted triple $J$ has coefficient $c_J=-2g_J+{c_J^{\rm fac}}$, where
\[
 {c_J^{\rm fac}}=\frac12\#\{n\in{\mathcal N_{\rm old}}:n\perp a,J,\ J\text{ independent}\}
 +\frac12\sum_{\substack{\{b,c\}\subset J\\n\in{\mathcal N_{\rm old}},\ n\perp a,b,c}}
       (1+\1_{b\perp c})(n\cdot d)^2,
\]
where $\{d\}=J\setminus\{b,c\}$.
The relevant old normals are $(e_i\pm e_j)/\sqrt2$, $1\le i<j\le3$.
For a companion--outside pair the common normal uses the other two coordinates; for two outside rays it uses their changed coordinates when two signs differ, their unchanged coordinates when one differs, and all three coordinate pairs when they are antipodal.
Two companions have none.
The remaining squared inner products, zero or $1/2$, give
\[
\begin{array}{@{}c|c@{}}
 J&c_J\\\hline
 \{e_1,e_2,b\},\ b\text{ outside}&-1/2\\
 \text{one outside parity triple}&-1/2\\
 \{e_i,b,c\},\ \{b,c\}\text{ counted by }{N_{{\rm diag},j}}&-\1_{i=j}/2\\
 \{e_i,b,c\},\ \{b,c\}\text{ counted by }{N_{{\rm edge},j}}&-\1_{i\ne j}/4\\
 \{e_i,b,c\},\ b,c\text{ antipodal}&0\\
 \text{a mixed outside triple}&0.
\end{array}
\]
For the last row, $g_J=1/2$ and ${c_J^{\rm fac}}=1$. The latter has nonzero contributions $1/2,1/2$ with an antipodal pair, and $1/2,1/4,1/4$ otherwise.
Summation gives the formula for $q_1$.
Each parity triple contains a unique ${N_{{\rm diag},3}}$ pair, so ${N_{\rm par}}\le R{N_{{\rm diag},3}}$. The other brackets are nonnegative, and $R(u+v)\ge2uv$ gives the bound.
\end{proof}

\begin{lemma}\label{lem:rooted-exchange-polynomial}
Fix a retained ray $a$ of weight $R$.
For a retained basis $B$ omitting $a$, let
\[
 f_B=\#\{b\in B:(B\setminus\{b\})\cup\{a\}\text{ is a basis}\}.
\]
The constant coefficient of ${Q}$ is
\begin{equation}\label{eq:weighted-rooted-exchange}
 4q_0
 =R\sum_{\substack{|J|=3\\J\cup\{a\}\text{ is a basis}}}x_J
  -\sum_{\substack{|B|=4,\ a\notin B\\B\text{ is a basis}}}(f_B-2)x_B.
\end{equation}
In particular, this coefficient depends only on the rooted linear matroid.
\end{lemma}
\begin{proof}
At zero core, squared basis determinants $1/4$ give the first sum as $4R{D_\perp}$ and ${\det H_-}=\frac14\sum_Bx_B$.
Deleting $b\in B$ gives a cofactor normal perpendicular to $a$ exactly when $(B\setminus\{b\})\cup\{a\}$ is dependent.
Each of these $4-f_B$ deletions contributes $x_B/4$ to $\sum_{e\perp a}\zeta_eh_e$, by the Gram determinant times squared-height formula.
Substitution in $q_0=R{D_\perp}-2{\det H_-}+\sum_{e\perp a}\zeta_eh_e$ proves the identity.
\end{proof}

For direct exchange calculations, write $a=\sum_{b\in B}\beta_b b$ in an unrooted basis $B$.
Cramer's rule gives
\[
 f_B=\#\{b\in B:\beta_b\ne0\}.
\]
Thus the exchange coefficient $-(f_B-2)$ in~\eqref{eq:weighted-rooted-exchange} can be read from the nonzero coordinates of the selected ray in that basis.

\subsection{One cube comparison for both configurations}

\begin{lemma}
\label{lem:full-companion-cube}\label{lem:cube-marginals}
In the coordinates of Lemma~\ref{lem:linear-core-frame}, set $R=u=v=1$.
Write $N_{\rm edge}=N_{{\rm edge},1}+N_{{\rm edge},2}$, $N_{\rm diag}=N_{{\rm diag},3}$, and $N_{\rm tri}$ for weights of edges, diagonals, and triples in the horizontal faces $\sigma_3=\pm1$.
Let $N_{\rm star}$ count weighted cube stars (a vertex and its three neighbors), and retain the parity-triple weight $N_{\rm par}$.
For arbitrary nonnegative outside weights,
\begin{align*}
 4{q_0}&=s+{N_{\rm edge}}+2{N_{\rm diag}}+{N_{\rm tri}}-{N_{\rm par}}-2{N_{\rm star}},
 \\
 {q_1}&=1+\frac{s}{2}+\frac{{N_{\rm edge}}}{4}+\frac{{N_{\rm diag}}}{2}-\frac{{N_{\rm par}}}{2}.
\end{align*}
If the outside weights lie in $[0,1]$, then
\begin{equation}
 N_{\rm edge}+2N_{\rm diag}-2N_{\rm par}\ge2(s-4)_+
 \quad\text{and}\quad
 4q_0\ge s+N_{\rm edge}+N_{\rm diag}-N_{\rm star}.
 \label{eq:cube-marginal-pair-bound}
\end{equation}
For binary outside weights, $s$ counts occupied vertices and
\begin{equation*}
 5{N_{\rm edge}}+10{N_{\rm diag}}+3{N_{\rm tri}}-7{N_{\rm par}}-6{N_{\rm star}}\ge5s-12.
\end{equation*}
Writing ${n_{\rm pair}}(s)=\binom{\lfloor s/2\rfloor}{2} +\binom{\lceil s/2\rceil}{2}$, one also has
\begin{equation*}
 4{q_0}\ge s+{n_{\rm pair}}(s)-\left\lfloor\frac{(s-2)_+}{2}\right\rfloor.
\end{equation*}
In particular, $4{q_0}\ge5,8,10$ for $s=4,5,6$, respectively, and ${q_1}\ge5$ for $s=6$.
\end{lemma}
\begin{proof}
\emph{Identities.}
Setting $R=u=v=1$ in Lemma~\ref{lem:linear-core-frame} gives the formula for $q_1$.
For $4q_0$, the rooted exchange formula~\eqref{eq:weighted-rooted-exchange} assigns coefficient one to a single outside ray, completed by both companions.
For an outside pair $b_\sigma,b_\tau$, let $r$ count the companions $e_i$, $i=1,2$, for which $\{a,e_i,b_\sigma,b_\tau\}$ is a basis, and put $B=\{e_1,e_2,b_\sigma,b_\tau\}$.
Its coefficient in $4q_0$ is $r-(f_B-2)$ when $B$ is a basis, and $r$ otherwise.
The following table lists all pair types, with representative $\sigma=(+,+,+)$. A dash means that $B$ is dependent.
\begin{center}
\small
\begin{tabular}{@{}c|cccc@{}}
pair type & $\tau$ & $r$ & $f_B$ & coefficient \\
\midrule
horizontal edge & $(-,+,+)$ & $1$ & --- & $1$ \\
horizontal diagonal & $(-,-,+)$ & $2$ & --- & $2$ \\
vertical edge & $(+,+,-)$ & $2$ & $4$ & $0$ \\
nonhorizontal diagonal & $(-,+,-)$ & $1$ & $3$ & $0$ \\
antipodal pair & $(-,-,-)$ & $0$ & $2$ & $0$ \\
\end{tabular}
\end{center}
Here horizontal means contained in a face $\sigma_3=\pm1$.
Thus the pair terms contribute exactly $N_{\rm edge}+2N_{\rm diag}$.

A triple without an antipodal pair is a face corner or a parity triple, with rooted contribution one.
A face corner admits a companion completion only in its remaining coordinate direction, with exchange number three, since the selected ray is the opposite pair's sum minus the signed companion.
Thus horizontal and nonhorizontal face triples have coefficients one and zero, respectively.
A parity triple has two companion completions of exchange number three, giving coefficient minus one. A triple with an antipodal pair contributes zero.

A four-ray basis containing an antipodal pair has exchange number two, since that pair sums to the selected ray.
Without such a pair, a four-set is a dependent square, a parity tetrahedron of zero weight, or a cube star.
For a cube star, sign changes reduce the calculation to the center $b_0=b_{(+,+,+)}$ and its neighbors
\[
 b_1=b_{(-,+,+)},\quad b_2=b_{(+,-,+)},\quad b_3=b_{(+,+,-)}.
\]
Since $b_i-b_0=-e_i$ for $i=1,2,3$, these four rays form a basis, with squared determinant $1/4$.
Moreover,
\[
 a=b_1+b_2+b_3-b_0.
\]
All four coefficients are nonzero, so $f_B=4$ and the star contributes $-2x_{b_0}x_{b_1}x_{b_2}x_{b_3}$ to $4q_0$.
This explains the term $-2N_{\rm star}$.
These contributions give the formula for $q_0$.

\emph{Comparisons.}
For weights in $[0,1]$, deletion increases products.
Deleting a star's center or vertical neighbor injects stars into parity triples or horizontal triples, respectively. Parity triples inject into their horizontal diagonals because each parity class retains at most three vertices.
Thus
\[
 N_{\rm star}\le N_{\rm par}\le N_{\rm diag},\quad
 N_{\rm star}\le N_{\rm tri},\quad N_{\rm par}\le2.
\]
For binary weights, the horizontal face sizes $n_+,n_-$ satisfy $n_++n_-=s$ and
\[
 N_{\rm edge}+N_{\rm diag}=\binom{n_+}2+\binom{n_-}2
 \quad\text{and}\quad N_{\rm tri}=\binom{n_+}3+\binom{n_-}3.
\]
The inequality $\binom{k}{2}\ge2k-3$ and multiaffinity therefore give $N_{\rm edge}+N_{\rm diag}\ge2s-6$ throughout $[0,1]^6$.
Hence
\[
 N_{\rm edge}+2N_{\rm diag}-2N_{\rm par}
 \ge N_{\rm edge}+N_{\rm diag}-N_{\rm par}\ge2s-8.
\]
Since $N_{\rm diag}\ge N_{\rm par}$ also makes the left side nonnegative, the weighted pair bound follows.
Subtracting the claimed weighted lower bound from the formula for $4q_0$ leaves $(N_{\rm diag}-N_{\rm par})+(N_{\rm tri}-N_{\rm star})\ge0$.

For binary weights, balancing $n_+,n_-$ minimizes the pair count, giving $N_{\rm edge}+N_{\rm diag}\ge n_{\rm pair}(s)$.
One star occupies four vertices, two occupy at least six, and $N_{\rm star}\le2$. Hence $N_{\rm star}\le\lfloor(s-2)_+/2\rfloor$, proving the cardinality bound and its stated values.
At $s=6$, the weighted pair bound and the formula for $q_1$ give $q_1\ge5$.

For the remaining binary inequality, note that for integral $0\le k\le4$,
\[
 g(k):=5\binom k2-3\binom k3-5k+6
 =1+\frac{(k-1)(k-2)(5-k)}2\ge1+(k-2)_+.
\]
For $k\le2$ the product is nonnegative; for $k\ge2$, subtracting $k-1$ leaves $(k-2)\{2-(k-3)^2\}/2\ge0$.
Each parity class supports at most one triple, and two triples require $s=6$.
Hence $2N_{\rm par}\le2+(s-4)_+$, so
\[
 g({n_+})+g({n_-})\ge2+({n_+}-2)_++({n_-}-2)_+
 \ge2+(s-4)_+\ge2{N_{\rm par}}.
\]
Together,
\begin{align*}
 &5{N_{\rm edge}}+10{N_{\rm diag}}+3{N_{\rm tri}}-7{N_{\rm par}}-6{N_{\rm star}}-5s+12\\
 \ge{}& 5{N_{\rm edge}}+10{N_{\rm diag}}-3{N_{\rm tri}}-7{N_{\rm par}}-5s+12\\
 ={}& g({n_+})+g({n_-})+5({N_{\rm diag}}-{N_{\rm par}})-2{N_{\rm par}} \ge 0.
\end{align*}
\end{proof}
The dense array has a corresponding spectral identity.
Let ${M_{\rm omit}}=\sum_{j=1}^3{\Pi_{m_j}}$, let ${G_{\rm omit}}$ be the Gram matrix of the three omitted unit normals, and set ${g_{\rm omit}}=\det {G_{\rm omit}}\ge0$.
Since $\operatorname{tr}{G_{\rm omit}}=3$ and $e_2({G_{\rm omit}})=9/4$, where $e_k$ is the $k$th elementary symmetric function of the eigenvalues, Sylvester's identity gives
\begin{equation}\label{eq:dense-frame-spectral-identities}
 \begin{gathered}
 \det(c{\Id}-{M_{\rm omit}})=c^2(c-3/2)^2-c{g_{\rm omit}},\\
 e_3(c{\Id}-{M_{\rm omit}})=4c^3-9c^2+\frac92c-{g_{\rm omit}}.
 \end{gathered}
\end{equation}
These two identities supply all dense-array endpoints below.
\begin{lemma}\label{lem:cube-constant-bound}
In either retained frame system, if every retained weight is at most $R$, then
\[
 q_0\ge\frac14Ruvs.
\]
After normalizing $R=1$ and specializing the other weights to zero or one, if exactly one companion is active then also $q_0\ge\frac12(s-4)_+$.
\end{lemma}
\begin{proof}
Normalize $R=1$ and use multiaffinity to take all weights binary.
Retain the cube notation of Lemma~\ref{lem:full-companion-cube}, so ${N_{\rm edge}}={N_{{\rm edge},1}}+{N_{{\rm edge},2}}$.
Let $N_{{\rm root},0}$ count outside triples completed to a basis by $a$, equivalently triples with independent sign vectors.
By the injections of Lemma~\ref{lem:cube-marginals}, $N_{{\rm root},0}\ge N_{\rm tri}+N_{\rm par}\ge2N_{\rm star}$.
With no active companion, the rooted exchange identity~\eqref{eq:weighted-rooted-exchange} gives
\[
 4q_0={N_{{\rm root},0}}-2{N_{\rm star}}\ge0.
\]
Indeed, a four-ray basis containing an antipodal pair has $f_B=2$ and contributes zero, while every other outside basis of positive weight is a cube star and has $f_B=4$.
If only the companion $e_i$ is active, let ${N_{{\rm root},i}}$ count outside pairs completed by $a,e_i$, and let ${N_{{\rm exc},i}}$ sum the negative exchange coefficients of bases consisting of $e_i$ and three outside rays.
Then
\[
 4q_0={N_{{\rm root},0}}+{N_{{\rm root},i}}-{N_{{\rm exc},i}}-2{N_{\rm star}},\quad\text{with } {N_{{\rm exc},i}}\le {N_{{\rm root},0}}.
\]
Projection away from $e_i$ leaves three distinct square corners. Their opposite pair represents $a$ with coefficients one, while the $e_i$ coefficient is zero or $\pm1$.
Thus exchange counts are two or three, and each count-three triple contributes to ${N_{{\rm root},0}}$.
Also ${N_{{\rm root},i}}\ge {N_{\rm diag}}+{N_{{\rm edge},j}}$ for the other companion coordinate $j\ne i$, and
\[
 {N_{\rm diag}}+{N_{{\rm edge},j}}\ge2{N_{\rm star}}.
\]
A nonfull horizontal face supports at most one star and then contains one diagonal and one edge in direction $j$.
A full face supplies two of each, while the other face has at most two vertices, allowing at most two stars.
This proves $q_0\ge0$.
For the stronger bound, the projected class sizes $k,s-k$ are at least one when $s=5$ and two when $s=6$, so $N_{{\rm root},i}=k(s-k)\ge4,8$, respectively.
With $N_{\rm star}\le1,2$, this gives $4q_0\ge2(s-4)_+$.

When both companions are active, the weighted lower bound for $q_0$ and ${N_{\rm star}}\le {N_{\rm diag}}$ give $4q_0\ge s+{N_{\rm edge}}\ge s$.
Interpolation completes the proof.
\end{proof}

\begin{lemma}\label{lem:maxray-mean-packets}
For a largest retained weight $R$, one has ${Q}\ge uvh(R+t)$.
If $R>0$ and $t\ge2R/3$, then
\[
 {Q}\ge uvh(R+h)+\frac{23}{324}R^4.
\]
\end{lemma}
\begin{proof}
The common cube bound, the frame identity, and the coefficients~\eqref{eq:common-high-core-coefficients} compare coefficientwise with $uv(t+s/4)(R+t)$, proving the first assertion.
Set $q_{\rm rem}=Q-uvh(R+h)$.
Its coefficients of $t,t^2,t^3$ are nonnegative: the first is at least $[R(u+v)-2uv]s/2$, and the other two are
\[
 R(u+v)-2uv+(3R-u-v)s/4\quad\text{and}\quad R-(u+v)/3.
\]
Monotonicity in $t$ and separate affinity in $u,v$ reduce to $t=2R/3$ and $u,v\in\{0,R\}$.
If $uv=0$, then ${q_{\rm rem}}={Q}\ge t^3q_3\ge16R^4/81$.
Otherwise scale to $R=u=v=1$. Separate concavity of ${q_{\rm rem}}$ in the outside weights reduces them to binary values.
The formulas for $q_0,q_1$ in Lemma~\ref{lem:full-companion-cube} give
\[
 q_{\rm rem}=\frac8{81}+\frac s9-\frac{s^2}{16}+q_{\rm corr},
\]
where
\[
12q_{\rm corr}=5N_{\rm edge}+10N_{\rm diag}+3N_{\rm tri}-7N_{\rm par}-6N_{\rm star}.
\]
Moreover, ${q_{\rm corr}}=(q_0-s/4)+(2/3)(q_1-1-s/2)\ge0$ by the constant-coefficient bound and the frame identity.
Therefore
\[
 {q_{\rm rem}}\ge\frac8{81}+\frac s9-\frac{s^2}{16}
                 +\max\left\{0,\frac{5s-12}{12}\right\}.
\]
For integer $s\le5$, concavity on $[0,2]$ and $[3,5]$ gives respective endpoint minima $23/324$ and $155/1296>23/324$.
At $s=6$, Lemma~\ref{lem:cube-marginals} gives $q_0\ge5/2$, $q_1\ge5$, and $q_{\rm corr}\ge5/3$, hence $q_{\rm rem}\ge59/324>23/324$.
Restore $R$ by homogeneity.
\end{proof}

\subsection{The zero-core bound}

\begin{proposition}
\label{prop:sharp-zero-core-exchange}
In either nonprincipal configuration, suppose that $t=0$ and that $R>0$ is a largest retained weight.
Then
\[
 81{Q}\ge10{D_H}.
\]
In particular, $9{Q}\ge{D_H}$.
\end{proposition}
\begin{proof}
Normalize $R=1$ and interpolate to binary retained weights. At $t=0$, write
\[
 q=4q_0=4{Q}\quad\text{and}\quad {N_{\rm bas}}=4{D_H}.
\]
If the active basis count ${N_{\rm bas}}$ is zero, the rooted exchange formula gives $q=0$. Assume henceforth ${N_{\rm bas}}>0$.

Let $s$ and $n\in\{0,1,2\}$ count active outside rays and companions.
If $N_{\rm anti}$ counts outside antipodal pairs, then
\begin{equation*}
 {N_{{\rm root},0}}=\binom{s}{3}-{N_{\rm anti}}(s-2),
\end{equation*}
since a triple of cube vectors is independent exactly when it has no antipodal pair.
Recall $N_{\rm star}=0$ for $s\le3$ and $N_{\rm star}\le1$ for $s\le5$.

With $m=1+n+s$ active rays, $\operatorname{tr}H=m$, so AM--GM gives
\begin{equation*}
 {N_{\rm bas}}\le\frac{m^4}{64}.
\end{equation*}
Also ${N_{\rm bas}}\le\binom m4$.

For $n=0$, $q={N_{{\rm root},0}}-2{N_{\rm star}}$. Rank forces $s\ge3$, with $q={N_{\rm bas}}=1$ at $s=3$.
For $s=4$, rank excludes two antipodal pairs, and one pair excludes stars. For $s=5$, $N_{\rm anti}\le2$ gives $N_{{\rm root},0}\ge4$ and $N_{\rm star}\le1$.
At $s=6$, the opposite-parity omitted vertices are adjacent or antipodal.
The respective triples $({N_{\rm anti}},{N_{\rm star}},{N_{{\rm root},0}})$ are $(2,2,12)$ and $(3,0,8)$, so $q=8$.

For $n=1$, choose companion $e_1$ and partition the outside rays by $\varepsilon_2\varepsilon_3$, with part sizes $n_{23,+},n_{23,-}$ positive by rank.
The exchange decomposition in Lemma~\ref{lem:cube-constant-bound} gives
\[
 q=N_{{\rm root},0}+N_{{\rm root},1}-N_{{\rm exc},1}-2N_{\rm star},\quad
 \text{where } N_{{\rm root},1}=n_{23,+}n_{23,-},
\]
with 
\[
N_{{\rm exc},1}\le\min\left\{N_{{\rm root},0},
 \left\lfloor\frac{n_{23,+}}2\right\rfloor n_{23,-}
 +\left\lfloor\frac{n_{23,-}}2\right\rfloor n_{23,+}\right\}.
\]
For the second bound, a count-three exchange uses a same-part face diagonal, and these diagonals form a matching.
Antipodal pairs have exchange count two, while pairs differing only at coordinate $1$ give dependence.
For $s\le3$, $q\ge {N_{{\rm root},1}}\ge s-1$.
At $s=4$, no star gives $q\ge3$, while one star has part sizes $2,2$ and gives $q\ge4-2=2$.
At $s=5$, no star gives $q\ge4$, while one star has part sizes $2,3$ and gives $q\ge6-2=4$.
The occupied frame sizes are $2,3,2$, so
\[
 e_2(H)=\binom22+\binom32+\binom22
       +\frac34(2\cdot3+2\cdot2+3\cdot2)=17.
\]
Maclaurin's inequality gives
\[
 {N_{\rm bas}}=4\det H\le\frac{e_2(H)^2}{9}
   =\frac{289}{9}<\frac{81}{10}\,4\le\frac{81}{10}q.
\]
At $s=6$, ${N_{{\rm root},0}}-2{N_{\rm star}}=8$, while part sizes $2,4$ or $3,3$ give ${N_{{\rm root},1}}\ge8$ and ${N_{{\rm exc},1}}\le8$, hence $q\ge8$.

For $n=2$, Lemma~\ref{lem:cube-marginals} gives $q\ge s$ and, at $s=4,5,6$, the sharper bounds $q\ge5,8,10$.
The table combines these lower bounds with ${N_{\rm bas}}\le\binom{1+n+s}{4}$. Its entries $64$ use the AM--GM trace bound, and $289/9$ is the Maclaurin estimate.
Full rank requires $s\ge3-n$, and there are six outside rays, so the table covers every case except $(n,s)=(2,6)$.
\begin{center}
\begin{tabular}{@{}ccc@{\quad}|ccc@{\quad}|ccc@{\quad}}
\multicolumn{3}{c}{$n=0$}&\multicolumn{3}{c}{$n=1$}&\multicolumn{3}{c}{$n=2$}\\
\cmidrule(lr){1-3}\cmidrule(lr){4-6}\cmidrule(lr){7-9}
$s$&$q\ge$&$N_{\rm bas}\le$&$s$&$q\ge$&$N_{\rm bas}\le$&$s$&$q\ge$&$N_{\rm bas}\le$\\
\midrule
3&1&1&2&1&1&1&1&1\\
4&2&5&3&2&5&2&2&5\\
5&2&15&4&2&15&3&3&15\\
6&8&35&5&4&$289/9$&4&5&35\\
 && &6&8&64&5&8&64
\end{tabular}
\end{center}
Every listed comparison gives $81q>10N_{\rm bas}$.

For the remaining fully occupied array, $H=3{\Id}-{M_{\rm omit}}$, so the dense-frame determinant identity~\eqref{eq:dense-frame-spectral-identities} gives
\[
 {N_{\rm bas}}=4\det H=81-12{g_{\rm omit}}\le81.
\]
Together with $q\ge10$, this proves $81q\ge10{N_{\rm bas}}$. Interpolation and homogeneity cover all weights.
\end{proof}

\subsection{Continuation from the two endpoints}

\begin{lemma}\label{lem:uniform-linear-core}
Let ${p_{\rm cmp}}(t)=9{Q}(t)-{D_H}(t)$ and write ${p_{{\rm cmp},1}}$ for its coefficient of $t$.
If every retained weight is at most the selected weight $R$, then
\[
 {p_{{\rm cmp},1}}\ge0.
\]
\end{lemma}
\begin{proof}
For $R>0$, normalize $R=1$ and reduce to binary weights. The case $R=0$ is immediate.
Let $n=u+v$ count active companions, $s$ active outside vertices, and $H_0=\sum_i x_i{\Pi_{a_i}}$.
Then ${p_{{\rm cmp},1}}=9q_1-e_3(H_0)$, and Maclaurin's inequality gives
\begin{equation*}
 e_3(H_0)\le\frac{(1+n+s)^3}{16}.
\end{equation*}

If $n=0$, let ${N_{\rm same}}$ count same-parity outside pairs, ${N_{\rm cross}^{a}}$ count cross-parity pairs independent with $a$, and ${N_{\rm mix}}$ count mixed outside triples.
The frame identity and the Gram determinants give
\[
 q_1=\frac{{N_{\rm same}}+{N_{\rm cross}^{a}}-{N_{\rm par}}}{2}\quad\text{and}\quad
 e_3(H_0)=\frac{{N_{\rm same}}+{N_{\rm cross}^{a}}+{N_{\rm mix}}}{2}+{N_{\rm par}}.
\]
Each parity class has at most three vertices, so ${N_{\rm par}}\le {N_{\rm same}}/3$ and ${N_{\rm mix}}\le3{N_{\rm same}}$.
Hence
\[
 {p_{{\rm cmp},1}}=4({N_{\rm same}}+{N_{\rm cross}^{a}})-\frac{11}{2}{N_{\rm par}}-\frac12{N_{\rm mix}}
 \ge\frac23{N_{\rm same}}+4{N_{\rm cross}^{a}}\ge0.
\]

For $n=1,2$, matching and cube marginal bounds, respectively, combined with the Maclaurin bound for $e_3(H_0)$, give
\[
 q_1\ge uv+\frac{s+(s-4)_+}{2}\quad\text{and}\quad
 p_{{\rm cmp},1}\ge9uv+\frac92\{s+(s-4)_+\}
                         -\frac{(1+n+s)^3}{16}.
\]
Indeed, for $n=1$, same-parity pairs whose unchanged coordinate is the inactive companion form a perfect matching, with at least $(s-4)_+$ occupied edges.
For $n=2$, substitute the weighted pair bound~\eqref{eq:cube-marginal-pair-bound} into the formula for $q_1$.
For $n=1$ the second bound is concave on $[1,4]$ and $[4,6]$, with positive endpoint values $45/16,9/2,4$. At $s=0$, $q_1=e_3(H_0)=0$.
For $n=2$ and $0\le s\le4$, concavity gives the endpoint minimum $\min\{117/16,89/16\}=89/16>0$.
At $s=5$, the same bounds give $q_1\ge4$ and $e_3(H_0)\le32$, hence ${p_{{\rm cmp},1}}\ge4>0$.
At $s=6$, $H_0=3\Id-{M_{\rm omit}}$. The dense-frame identity~\eqref{eq:dense-frame-spectral-identities} and the cube marginal bound give $e_3(H_0)\le81/2$ and $q_1\ge5$, hence ${p_{{\rm cmp},1}}\ge9/2>0$.
Interpolation and homogeneity prove the result.
\end{proof}

\begin{lemma}\label{lem:common-core-endpoint}
At binary retained weights with $R=1$, one has $9{Q}(2/3)>{D_H}(2/3)$.
\end{lemma}
\begin{proof}
Write $n=u+v\in\{0,1,2\}$ for the active companion count, $s$ for the outside count, and $h=2/3+s/4$.
The diagonal of $H$ is $(1+h,u+h,v+h,h)$, so Hadamard's inequality gives
\[
 {D_H}(2/3)\le {\operatorname{upper}_n(s)}:=(1+h)^{n+1}h^{3-n}.
\]
The coefficient and cube bounds give
\[
 9Q(2/3)\ge\operatorname{lower}_n(s):=
 \begin{cases}
  \frac83+s+s^2,&n=0,\\
  \frac{52}{9}+5s+\frac{15}{2}(s-4)_+,&n=1,\\
  \frac{98}{9}+\frac{25s}{4}+\frac{15}{4}(s-3)_+,&n=2.
 \end{cases}
\]
For $n=0$, use the preceding pair counts $N_{\rm same},N_{\rm cross}^{a}$ and parity sizes $n_{{\rm par},\pm}\le3$.
Dependent cross-parity pairs form a matching, so $N_{\rm cross}^{a}\ge n_{{\rm par},+}n_{{\rm par},-} -\min(n_{{\rm par},+},n_{{\rm par},-})$.
Since $N_{\rm par}\le N_{\rm same}/3$,
\[
 2q_1={N_{\rm same}}-{N_{\rm par}}+{N_{\rm cross}^{a}}
 \ge\frac23{N_{\rm same}}+{n_{{\rm par},+}}{n_{{\rm par},-}}-\min({n_{{\rm par},+}},{n_{{\rm par},-}})\ge\frac{s(s-2)}3.
\]
For the last inequality, assume ${n_{{\rm par},+}}\ge {n_{{\rm par},-}}$. Three times the difference is ${n_{{\rm par},+}}{n_{{\rm par},-}}+{n_{{\rm par},+}}-2{n_{{\rm par},-}}\ge0$ for these nonnegative integers.
Together with $q_0\ge0$, $q_2=3s/4$, and $q_3=1$, this proves the $n=0$ bound.
For $n=1$, use $q_0\ge(s-4)_+/2$, the matching bound $q_1\ge\{s+(s-4)_+\}/2$ from the proof of Lemma~\ref{lem:uniform-linear-core}, and $q_2=1+s/2$, $q_3=2/3$.

For $n=2$, the cube identities with the large-core correction $q_{\rm corr}$ give
\[
 Q(2/3)=\frac{98}{81}+\frac{25s}{36}+q_{\rm corr},\quad
 \text{with}\quad q_{\rm corr} \ge\max\{0,(5s-12)/12\}\ge\frac5{12}(s-3)_+.
\]

For $n=0$, set $z=s/4\in[0,3/2]$. Expansion gives
\[
 {\operatorname{lower}_0}(s)-{\operatorname{upper}_0(s)}
 =\frac{176}{81}+\frac{40z}{27}
       +\left(\frac{34}{3}-\frac{11z}{3}-z^2\right)z^2>0,
\]
because the bracket is at least $43/12$ on this interval.
For $n=1,2$, convexity of ${\operatorname{upper}_n}$ and piecewise affinity of ${\operatorname{lower}_n}$ make ${\operatorname{lower}_n}-{\operatorname{upper}_n}$ concave between breakpoints.
Its values exceed $4,6,3$ at $s=0,4,6$ for $n=1$, and $7,9,2$ at $s=0,3,5$ for $n=2$.
This proves the comparison on $[0,6]$ and $[0,5]$, respectively.

At the remaining integral value $n=2,s=6$, Lemma~\ref{lem:cube-marginals} gives $q_0\ge5/2$ and $q_1\ge5$, hence ${q_{\rm corr}}\ge5/3$ and $9{Q}(2/3)\ge1141/18$.
The dense-frame identity \eqref{eq:dense-frame-spectral-identities}, with $c=11/3$, gives
\[
 {D_H}(2/3)\le c^2(c-3/2)^2=\frac{20449}{324}.
\]
The difference is at least $89/324>0$.
\end{proof}

\begin{lemma}\label{lem:maxray-low-core}
If $R>0$ is a largest retained weight and $0\le t\le2R/3$, then $9{Q}\ge{D_H}$.
\end{lemma}
\begin{proof}
Normalize $R=1$ and reduce to binary weights.
For the comparison polynomial ${p_{\rm cmp}}(t)=9{Q}(t)-{D_H}(t)$, write
\[
 {p_{\rm cmp}}(t)={p_{{\rm cmp},0}}+{p_{{\rm cmp},1}}t+{p_{{\rm cmp},2}}t^2
                         +{p_{{\rm cmp},3}}t^3-t^4.
\]
Lemma~\ref{lem:uniform-linear-core} gives ${p_{{\rm cmp},1}}\ge0$.
Let $n=u+v$, with ${N_{\rm same}}$ counting same-parity outside pairs and ${N_{\rm cross}}$ counting all cross-parity pairs, including those dependent with $a$.
Cauchy--Binet and the coefficients~\eqref{eq:common-high-core-coefficients} give
\[
 {p_{{\rm cmp},3}}=8-4n-s\quad\text{and}\quad
 {p_{{\rm cmp},2}}=8n-10uv+(6-3n)s-{N_{\rm same}}-\frac34{N_{\rm cross}}.
\]
If ${p_{{\rm cmp},3}}>0$, then $n\le1$ and $uv=0$.
Since ${N_{\rm same}}\le s$ and ${N_{\rm cross}}\le3s/2$, we have ${p_{{\rm cmp},2}}\ge8n+(31/8-3n)s\ge0$.
Thus ${p_{{\rm cmp},2}}<0$ implies ${p_{{\rm cmp},3}}\le0$.
The derivative ${p_{\rm cmp}}'$ therefore has at most one coefficient sign change. Descartes' rule gives at most one positive critical point, necessarily a maximum, so the minimum is at an endpoint.
Proposition~\ref{prop:sharp-zero-core-exchange} gives ${p_{\rm cmp}}(0)\ge0$, and Lemma~\ref{lem:common-core-endpoint} gives ${p_{\rm cmp}}(2/3)>0$.
Interpolation and homogeneity finish the proof.
\end{proof}

\begin{proof}[Proof of Lemma~\ref{lem:maxray-mean-comparison}]
The assertions are Lemmas~\ref{lem:maxray-low-core} and \ref{lem:maxray-mean-packets}, respectively.
\end{proof}

\section{Sharp stability of the 24-cell}
\label{sec:local-variation}

The direct entropy estimate already controls Hausdorff distance from the regular cell.
We now derive the local quadratic cost at $D_4$.
The essential geometric step is its sharp comparison with support displacements, which determines the optimal stability coefficient and the extremal deformation.
Recall that $\operatorname{def}(P)=\mathcal I(P)/2^{11/4}-1$ is the normalized area excess and $d_{24}(P)$ compares centered unit-volume bodies modulo orthogonal maps.
We follow the three-dimensional argument of Song~\cite{Song2026Local} in separating metric and affine strain, using polar exponential coordinates, computing mixed second-order terms, and comparing these coordinates with Hausdorff distance in both directions.

\subsection{Joint second variation at \texorpdfstring{$D_4$}{D4}}

Use the cells $\mathcal C_H,V_H$ and facet weights of Lemma~\ref{lem:universal-cofactors} with $\Lambda=D_4$.
For $|\det L|=1$, write $L^{-\mathsf T}=Ue^Y$, where $U$ is orthogonal and $Y\in\Sym_0(\R^4)$.
The logarithm $Y$ records affine strain after the orthogonal factor is removed.

\begin{proposition}\label{prop:local-isoperimetric-cost}
Let $H=\Id+X>0$, with $X\in\Sym_0(\R^4)$.
Uniformly across adjacent Voronoi chambers, as $r_{\rm par}=\|X\|+\|Y\|\to0$,
\[
  \log\frac{\mathcal I(LV_H)}{2^{11/4}}
  =\frac1{20}\|X\|^2+\frac5{24}\|Y-X/10\|^2
  +O(r_{\rm par}^3).
\]
For small $X$, the minimum over $|\det L|=1$ is attained at a unique polar logarithm $Y_H$, with
\[
  \inf_{|\det L|=1}\log\frac{\mathcal I(LV_H)}{2^{11/4}}
  =\frac1{20}\|X\|^2+O(\|X\|^3)
  \quad\text{and}\quad
  Y_H=X/10+O(\|X\|^2).
\]
The optimized coefficient $1/20$ is sharp.
\end{proposition}
\begin{proof}
Let $\varepsilon=\|X\|$.
The inclusions $(1-\varepsilon)C_{24}\subseteq\mathcal C_H\subseteq(1+\varepsilon)C_{24}$ and the unit circumradius of $C_{24}$ bound every facet vector by $\|v\|\le2(1+\varepsilon)/(1-\varepsilon)$.
Squared norms in $D_4$ are even integers, so for small $\varepsilon$ only roots and vectors of squared norm four can occur.
All longer-vector inequalities are then strictly slack throughout $(1+\varepsilon)C_{24}$.
The vectors of squared norm four are $\pm2e_i$ and those with four coordinates $\pm1$, each exposing a vertex of $C_{24}$.
For a point $x$ of such a perturbed facet, $x/(1+\varepsilon)\in C_{24}$ has supporting slack $O(\varepsilon)$.
Its vertex representation assigns total weight $O(\varepsilon)$ away from the exposed vertex, so the facet has diameter $O(\varepsilon)$ and cofactor weight $O(\varepsilon^3)$.

For an old root facet, translation by $-Hv/2$ gives inequalities $u\cdot y\le(u^{\mathsf T}Hu-u^{\mathsf T}Hv)/2$ in $v^\perp$.
For the finite candidate set and $u\ne0,v$, their right sides are positive at $\Id$ and vary by relative $O(\varepsilon)$.
Comparison with $(1\pm C\varepsilon)$ dilates of the old octahedron gives $w_e=1/3+O(\varepsilon)$.
These estimates are uniform across chambers. Give absent facets weight zero.

Write $\operatorname{surf}(X,Y)=\mathcal I(LV_H)/2^{11/4}$.
The surface formula~\eqref{eq:universal-surface-formula} has prefactor $D_H^{-7/8}/4$ and marks $\kappa_e=1$ on roots and $\kappa_e=2$ on new normals.
Since new weights are $O(\varepsilon^3)$, replacing all attenuation factors by one changes $\operatorname{surf}$ by $O(r_{\rm par}^3)$.
For $a=e^{\mathsf T}Xe$ and $b=e^{\mathsf T}Ye$, the surface factor expands as
\[
 \|e^YH^{1/2}e\|
 =1+\frac12a+b+\frac12e^{\mathsf T}(XY+YX)e+e^{\mathsf T}Y^2e
 -\frac18a^2-\frac12ab-\frac12b^2+O(r_{\rm par}^3).
\]
Cofactor balance, traced against $\Id$, $X$, and $Y$, gives
\begin{align*}
 \sum_e w_e&=4-\|X\|^2+O(\varepsilon^3),\\
 \sum_e w_e a&=-\|X\|^2+O(\varepsilon^3),\\
 \sum_e w_e b&=-\langle X,Y\rangle+O(\varepsilon^2\|Y\|).
\end{align*}
Use $w_e^0=1/3$ in the quadratic terms, the root moments of Lemma~\ref{lem:frame-geometry}, and $D_H^{-7/8}=1+7\|X\|^2/16+O(\varepsilon^3)$ to obtain
\[
 \log\operatorname{surf}(X,Y)
 =\frac5{96}\|X\|^2+\frac5{24}\|Y\|^2
 -\frac1{24}\langle X,Y\rangle+O(r_{\rm par}^3).
\]
Completing the square proves the joint expansion.
Its first two $Y$-derivative remainders are $O(r_{\rm par}^2)$ and $O(r_{\rm par})$ because the weights are $Y$-independent, the finite surface factors are analytic, and old-weight errors multiply quadratic terms.

At $X=0$, cofactor balance gives $\operatorname{surf}(0,Y)\ge\tr(e^Y)/4\ge1$, with equality only at $Y=0$.
For small $X$, the old normal matrix $\sum_{e\in\mathcal N_{\rm old}}w_e\sqrt{h_e}\Pi_{n_e}$ tends to $\Id$, so $\|e^Yn_e\|\ge n_e^{\mathsf T}e^Yn_e$ gives the uniform bound $\operatorname{surf}(X,Y)\ge c\tr(e^Y)$.
This is coercive on $\tr Y=0$, so minima exist. Uniform convergence on bounded sets forces every minimizer to zero as $X\to0$.
There the $Y$-Hessian tends uniformly to $5\Id/12>0$ on $\Sym_0(\R^4)$, giving uniqueness.
The gradient equation first gives $Y_H=O(\varepsilon)$ and then $Y_H=X/10+O(\varepsilon^2)$. Substitution proves the optimized formula and sharpness.
\end{proof}

\subsection{Local coordinates and sharp Hausdorff stability}
\label{sec:four-dimensional-stability}

The next lemma recovers metric and affine parameters from Hausdorff proximity.
It converts the cubic error in the local expansion into a cubic Hausdorff remainder and yields the reverse bound $\operatorname{def}(P)\le Cd_{24}(P)^2$.

\begin{lemma}\label{lem:24-cell-local-converse}
Let $P_0=\widehat C_{24}$ be the centered unit-volume regular $24$-cell.
There is $C>0$ such that every centered unit-volume four-dimensional parallelohedron $P$ sufficiently close to $P_0$ admits $P=\widehat{LV_H}$ for some $H$ in a nonprincipal cone such that
\[
 \|H-\Id\|+\|Y\|\le C d_H(P,P_0),\quad
 \text{with}\quad \tr H=4,\;\; |\det L|=1,\;\; L^{-\mathsf T}=Ue^Y.
\]
\end{lemma}

\begin{proof}
The cell $C_{24}$ has no nonzero segment summand.
Exposing a vertex of the projection of $Q+[-z,z]$, $z\ne0$, to $z^\perp$ gives an edge parallel to $z$.
Every two-face containing that edge has $[-z,z]$ as a summand and hence two parallel edges, contradicting the triangular two-faces of $C_{24}$.

Let $P_j\to P_0$ and set $\widetilde P_j=2^{1/4}P_j\to C_{24}$.
Classification allows a fixed cone type along a subsequence.
The principal type is impossible: its at most ten centered segment summands lie in $\widetilde P_j$, so their bounded generators would converge to a zonotope representation of $C_{24}$.

In either nonprincipal type, Lemma~\ref{lem:facet-cofactors} gives
\[
  \widetilde P_j
  =A_j\left(t_jC_{24}+\frac12\sum_{i=1}^9x_{ji}[-a_i,a_i]\right),
  \quad\text{with}\quad t_j,x_{ji}\ge0,
\]
where $A_j$ is invertible and the ray configuration is fixed.
Let $T_j=t_jA_j$ and $z_{ji}=x_{ji}A_ja_i$.
The summands lie in $\widetilde P_j$, so $T_j,z_{ji}$ are bounded and their subsequential limits satisfy
\[
  C_{24}=TC_{24}+\frac12\sum_{i=1}^9[-z_i,z_i].
\]
The absence of segment summands forces $z_i=0$ and $TC_{24}=C_{24}$. The vertex identity $\sum_{v\in\operatorname{vert}C_{24}}vv^{\mathsf T}=6\Id$ makes $T$ orthogonal.
Thus $t_j>0$ eventually and $x_{ji}/t_j=\|z_{ji}\|/\|T_ja_i\|\to0$.
Consequently every sufficiently close $P$ satisfies $2^{1/4}P=\operatorname{cell}(T,y)$ with $T$ near a symmetry and $y$ small, where
\[
 \operatorname{cell}(T,y)=T\left(C_{24}+\frac12\sum_i y_i[-a_i,a_i]\right),
 \quad\text{with}\quad y_i\ge0.
\]
We claim, near $(\Id,0)$,
\begin{equation}\label{eq:24-cell-inverse-support}
 \|T-\Id\|+\sum_i y_i
 \le C d_H\bigl(\operatorname{cell}(T,y),C_{24}\bigr).
\end{equation}
Otherwise, some $T_j\to \Id$, $y_j\to0$ satisfy
\[
 \varepsilon_j=\|T_j-\Id\|+\sum_i y_{ji}>0
 \quad\text{and}\quad d_H\bigl(\operatorname{cell}(T_j,y_j),C_{24}\bigr)=o(\varepsilon_j).
\]
Normalize $(T_j-\Id,y_j)$ by $\varepsilon_j$ and pass to a limit $(B,z)$, with $z_i\ge0$ and $\|B\|+\sum_i z_i=1$.
For fixed $u$ in the interior $N_v^\circ$ of a vertex normal cone, eventually $T_j^{\mathsf T}u\in N_v^\circ$, so the normalized support identity gives
\begin{equation*}
 u\cdot Bv+\frac12\sum_i z_i|a_i\cdot u|=0,
 \quad\text{with}\quad  u\in N_v^\circ.
\end{equation*}
For each retained $a_i$, its omitted frame vector $m_k$ lies in $N_{m_k}^\circ$ and is perpendicular to $a_i$.
Along $u=m_k+sa_i$, the $i$th term has derivative jump $z_i$ at zero, all other absolute-value terms have nonnegative jumps, and the linear term has none.
Hence $z=0$, and then $Bv=0$ for every vertex, contradicting normalization.

Left multiplication by the nearby symmetry $T_0^{-1}$ gives the same estimate, uniformly over the finitely many symmetries and two configurations.
Set
\[
 H^0=\Id+\sum_i y_i\Pi_{a_i},\quad
 H=\frac{4H^0}{\tr H^0},\quad
 L=\frac{T(H^0)^{1/2}}{|\det(T(H^0)^{1/2})|^{1/4}}.
\]
The Minkowski formula gives $P=\widehat{LV_H}$, and smoothness of the polar logarithm bounds $\|H-\Id\|+\|Y\|$ by $C(\|T-T_0\|+\sum_i y_i)$, as required.
\end{proof}

The metric coefficient $1/20$ can be converted into an exact intrinsic coefficient.
To picture its extremal deformation, sweep the regular cell along three perpendicular segments, with lengths in the ratios $6:5:5$, and leave the fourth direction unswept.
These Minkowski additions replace the two vertices on the fourth axis by facets shaped as small rectangular boxes.
Metric conversion is followed by an affine strain that contracts the distinguished direction and expands the other three, and then by volume normalization.
The inward retreat along the direction of the longest sweep determines the leading Hausdorff displacement.

\begin{theorem}
\label{thm:sharp-Hausdorff-coefficient}
There are constants $C,r>0$ such that every four-dimensional parallelohedron with $d_{24}(P)<r$ satisfies
\begin{equation}\label{eq:sharp-Hausdorff-cubic}
 \operatorname{def}(P)\ge
 \frac{\sqrt2}{8}d_{24}(P)^2-Cd_{24}(P)^3.
\end{equation}
The coefficient is optimal:
\begin{equation*}
 \lim_{r\downarrow0}\inf_{0<d_{24}(P)\le r}
 \frac{\operatorname{def}(P)}{d_{24}(P)^2}
 =\frac{\sqrt2}{8}.
\end{equation*}
An extremal family is
\[
 K_{\rm ext} =\operatorname{diag}(12,10,10,0),\quad
 Y_{\rm ext} =\operatorname{diag}(4,-1,-1,-2),\quad
 P_\tau=\widehat{e^{-\tau Y_{\rm ext}}V_{\Id+\tau K_{\rm ext}}},
\]
with $\tau>0$.
For $0<r_P:=d_{24}(P)<r$, define
\[
 \varepsilon(P)=\max\left\{
 \frac{\operatorname{def}(P)}{r_P^2}-\frac{\sqrt2}{8},0\right\}
 \quad\text{and}\quad \tau=\frac{r_P}{8\cdot 2^{-1/4}}.
\]
Then, after increasing $C$ if necessary,
\begin{equation*}
 \inf_{U\in \mathrm{O}(4)}d_H(\widehat P,UP_\tau)
 \le C r_P\sqrt{\varepsilon(P)+r_P}.
\end{equation*}
Hence every sequence with $d_{24}(P_j)\to0$ and $\operatorname{def}(P_j)/d_{24}(P_j)^2\to\sqrt2/8$ lies within $o(d_{24}(P_j))$ of this family, up to rotation.
\end{theorem}

The bound on distance to the extremal family is relative to the total departure $r_P$ from the regular cell. An excess normalized cost $\varepsilon(P)$ permits only an $O(\sqrt{\varepsilon(P)+r_P})$ relative departure from that family.

\begin{proof}
Let $c_{\rm vol}=2^{-1/4}$, so $\widehat C_{24}=c_{\rm vol}C_{24}$.
In either nonprincipal cone, write
\[
 K=\sum_i z_i\Pi_{a_i},\quad s=\sum_i z_i,\quad X=\tf K,
 \quad\text{with }z_i\ge0,\quad Y\in\Sym_0(\R^4).
\]
After rotation, every sufficiently close unit-volume parallelohedron has the form
\begin{equation*}
 P(z,Y)=c_{\rm vol}\det(\Id+K)^{-1/8}e^{-Y}(\Id+K)^{-1/2}
 \left(C_{24}+\frac12\sum_i z_i[-a_i,a_i]\right).
\end{equation*}
Lemma~\ref{lem:24-cell-local-converse} gives such coordinates with
\begin{equation}\label{eq:intrinsic-small-coordinates}
 r_{\rm par}:=\|z\|+\|Y\|\le C_1d_{24}(P).
\end{equation}
The traceless projector blocks are orthogonal, with Gram matrix $G=\Id_3-\frac14\1\1^{\mathsf T}>0$ on each retained triple.
Thus the local converse controls the weights, and division by the core coefficient gives the representation.
Trace normalization replaces $X$ by $X/(1+s/4)=X+O(r_{\rm par}^2)$, so the second variation gives
\begin{equation}\label{eq:intrinsic-quadratic-energy}
 \operatorname{def}(P(z,Y))=\operatorname{cost}(z,Y)+O(r_{\rm par}^3),
\end{equation}
where
\[
\operatorname{cost}(z,Y)=\frac1{20}\|X\|^2+
              \frac5{24}\|Y_{\rm res}\|^2
 \quad\text{and}\quad Y_{\rm res}=Y-X/10.
\]
The positive frame Gram matrices make $\operatorname{cost}$ positive definite.

\smallskip\noindent
\emph{Support displacement.}
Let $N_v$ be the closed normal cone at a vertex $v$ of $C_{24}$.
Define
\begin{gather*}
 B=-Y-\frac12K-\frac{s}{8}\Id
   =-Y_{\rm res}-\frac35K-\frac{s}{10}\Id,\qquad
 g_z(u)=\frac12\sum_i z_i|a_i\cdot u|,\\
 \operatorname{disp}(z,Y)=\max_{v\in\operatorname{vert}C_{24}}
       \sup_{u\in N_v,\,\|u\|=1}|u\cdot Bv+g_z(u)|.
\end{gather*}
The closed cones retain every incident vertex at a wall.
We claim
\begin{equation}\label{eq:intrinsic-support-upper}
 d_H(P(z,Y),c_{\rm vol}C_{24})\le c_{\rm vol}\operatorname{disp}(z,Y)+O(r_{\rm par}^2).
\end{equation}
The prefactor divided by $c_{\rm vol}$ is $\Id+B+O(r_{\rm par}^2)$, and $g_z((\Id+B)^{\mathsf T}u)=g_z(u)+O(r_{\rm par}^2)$ for unit $u$.
With an active vertex $v_t$ along $w_t=u+tB^{\mathsf T}u$, absolute continuity of the piecewise linear support function gives
\[
 h_{C_{24}}((\Id+B)^{\mathsf T}u)-h_{C_{24}}(u)+g_z(u)
 =\int_0^1\bigl(u\cdot Bv_t+g_z(u)\bigr)\,dt.
\]
Since $w_t\in N_{v_t}$, the integrand has absolute value at most $\|w_t\|\operatorname{disp}(z,Y)+O(r_{\rm par}^2)=\operatorname{disp}(z,Y)+O(r_{\rm par}^2)$.
Taking the supremum over unit $u$ gives the claimed Hausdorff bound.

\smallskip\noindent
\emph{Sharp comparison.}
We prove
\begin{equation*}
 \operatorname{disp}(z,Y)^2\le8\operatorname{cost}(z,Y).
\end{equation*}
Orient the frame containing $v$ so that $v=e_1$, and write $u=(h,b)\in N_v$. The coordinate and half-sign vertices give $h\ge\|b\|_1$.
The functional $u\cdot Bv+g_z(u)$ is
\[
 -u\cdot Y_{\rm res}v+\sum_i\ell_i z_i,
 \quad\text{where }\ell_i=\frac12|a_i\cdot u|
       -\frac35(a_i\cdot u)(a_i\cdot v)-\frac1{10}h.
\]
Because $G^{-1}=\Id_3+\1\1^{\mathsf T}$ and the frame blocks are orthogonal, its squared dual norm for $\operatorname{cost}$ is
\begin{equation*}
 \operatorname{dual}_v(u)=\frac{24}{5}\left\|\tf
       \frac{uv^{\mathsf T}+vu^{\mathsf T}}2\right\|^2
       +20\sum_{j=1}^3\ell_{T_j}^{\mathsf T}G^{-1}\ell_{T_j}.
\end{equation*}
The first term is $(18h^2+12\|b\|^2)/5$, and the dual norm also bounds the functional on $z_i\ge0$.
In either other frame, orient $a=(1,\varepsilon)/2$. Then $\ell_a=\varepsilon\cdot b/10$, and orthogonality gives
\[
 \sum_{a\in\mathrm{fr}_j}\ell_a=0\quad\text{and}\quad
 \sum_{a\in\mathrm{fr}_j}\ell_a^2=\frac{\|b\|^2}{25}.
\]
Omitting any one vector therefore leaves $\ell_{T_j}^{\mathsf T}G^{-1}\ell_{T_j}=\|b\|^2/25$.

If $v$ is retained, write $(p,q,r)$ for the absolute coordinates of $b$, with $r$ along the omitted vector.
The remaining block has coefficients $(-h/5,-h/10+p/2,-h/10+q/2)$, giving
\[
 \operatorname{dual}_v(u)=8h^2-10h(p+q)+14(p^2+q^2)+10pq+4r^2.
\]
Since $h\ge p+q+r$,
\[
 8\|u\|^2-\operatorname{dual}_v(u)
 \ge4(p^2+q^2)+10pq+10r(p+q)+4r^2\ge0,
\]
with equality only when $b=0$.
If $v$ is omitted, let $\ell=\|b\|_1$.
Its own-frame coefficients are $-h/10+|b_i|/2$, and the same calculation gives
\[
 \operatorname{dual}_v(u)=6h^2-8h\ell+9\|b\|^2+5\ell^2
 \le6\|u\|^2,
\]
since $h\ge\ell\ge\|b\|$.
Cauchy--Schwarz proves the sharp comparison, and \eqref{eq:intrinsic-support-upper}, \eqref{eq:intrinsic-small-coordinates}, and \eqref{eq:intrinsic-quadratic-energy} give
\[
 d_{24}(P)^2\le8c_{\rm vol}^2\operatorname{def}(P)+O(d_{24}(P)^3),
\]
proving the sharp lower bound.

\smallskip\noindent
\emph{Extremal family.}
Equality requires a retained $v$ and $u=v$. Choose $v=a_1=e_1$ and omitted vector $e_4$.
Dual Cauchy--Schwarz and $z_i\ge0$ force a positive multiple of
\[
 (z_1,z_2,z_3)=(12,10,10),\quad z_i=0\ (i>3),\quad
 Y_{\rm res}=\frac{24}{5}\tf \Pi_{e_1}.
\]
This gives $K_{\rm ext},Y_{\rm ext}$, with $\operatorname{cost}=8$ and radial displacement $-8$ at $u=v=e_1$.
Thus $\operatorname{disp}=8$ and
\begin{equation*}
 \operatorname{def}(P_\tau)=8\tau^2+O(\tau^3)\quad\text{and}\quad
 d_{24}(P_\tau)=8c_{\rm vol}\tau+O(\tau^2).
\end{equation*}
A minimizing rotation satisfies $d_H(U_\tau C_{24},C_{24})=O(\tau)$. After right multiplication by the inverse of the nearby symmetry, \eqref{eq:24-cell-inverse-support} with no segments gives $U_\tau=\Id+O(\tau)$.
Since $e_1$ exposes a unique vertex of $C_{24}$,
\[
 h_{U_\tau C_{24}}(e_1)=e_1\cdot U_\tau e_1=1+O(\tau^2)
 \quad\text{and}\quad h_{P_\tau}(e_1)=c_{\rm vol}(1-8\tau)+O(\tau^2).
\]
Thus rotations preserve the leading radial displacement, proving the asymptotics and optimality.

\smallskip\noindent
\emph{Quantitative rigidity.}
Recall that $\operatorname{cost}$ is the quadratic area cost and $\operatorname{disp}$ the first-order support displacement.
For $\xi=(z,Y_{\rm res})$, use the norm $\|\xi\|_{\mathrm{cost}}=\sqrt{\operatorname{cost}(\xi)}$ and let $\operatorname{modes}$ be the finite set of positive extremal vectors just identified, indexed by retained vertices.
Each has $\operatorname{cost}=8$.
We claim that
\begin{equation}\label{eq:quantitative-tangent-rigidity}
 \min_{\xi_{\rm ext}\in\operatorname{modes}}
 \left\|\xi-\sqrt{\operatorname{cost}(\xi)/8}\,\xi_{\rm ext}\right\|_{\mathrm{cost}}^2
 \le C\{8\operatorname{cost}(\xi)-\operatorname{disp}(\xi)^2\}.
\end{equation}
Normalize $\operatorname{cost}(\xi)=1$ and set $\gamma=8-\operatorname{disp}(\xi)^2$.
The left side is at most $4$, so only small $\gamma$ matters.
Choose $v,u$ attaining $\operatorname{disp}$ and let $r_u$ represent the functional $\xi\mapsto u\cdot Bv+g_z(u)$ in the $\operatorname{cost}$ inner product. Thus $\|r_u\|_{\mathrm{cost}}^2=\operatorname{dual}_v(u)$.
For $\gamma<1$, the bound $\operatorname{dual}_v(u)\le6$ excludes omitted vertices.
At a retained vertex the dual estimate gives
\[
 7\le \operatorname{disp}^2\le \operatorname{dual}_v(u)\le8
 \quad\text{and}\quad 4\|b\|^2\le8-\operatorname{dual}_v(u)\le\gamma.
\]
Lipschitz dependence on $u$, $\|u-v\|\le\sqrt2\|b\|$, and $r_v=-\xi_{\mathrm{ext},v}$ give
\[
 \left\|\frac{r_u}{\sqrt{\operatorname{dual}_v(u)}}+
                  \frac{\xi_{\mathrm{ext},v}}{\sqrt8}\right\|_{\mathrm{cost}}\le C\sqrt\gamma.
\]
If $\sigma$ is the sign of $\langle\xi,r_u\rangle_{\mathrm{cost}}$, the exact Cauchy--Schwarz remainder gives
\[
 \left\|\xi-\sigma\frac{r_u}{\sqrt{\operatorname{dual}_v(u)}}\right\|_{\mathrm{cost}}^2
 =2\left(1-\frac{\operatorname{disp}}{\sqrt{\operatorname{dual}_v(u)}}\right)
 \le\frac{2\gamma}{7}.
\]
Consequently $\|\xi+\sigma \xi_{\mathrm{ext},v}/\sqrt8\|_{\mathrm{cost}}\le C\sqrt\gamma$.
If $\sigma=1$, the nonnegative $v$-coordinate differs from $-12/\sqrt8$ by at least $12/\sqrt8$, and its squared dual norm $20(G^{-1})_{11}=40$ forces distance at least $3/\sqrt{20}$.
For small $\gamma$ this is impossible, so $\sigma=-1$ and the rigidity estimate follows.

For the coordinates of $P$, write $r_P=d_{24}(P)>0$.
Equations~\eqref{eq:intrinsic-small-coordinates}--\eqref{eq:intrinsic-support-upper} give
\[
 8\operatorname{cost}-\operatorname{disp}^2\le8\left\{\operatorname{def}(P)
                    -\frac{\sqrt2}{8}r_P^2\right\}+Cr_P^3
 \le C r_P^2\{\varepsilon(P)+r_P\}.
\]
The local formula is uniformly Lipschitz, and the frame changes and coordinate permutations identifying $\operatorname{modes}$ preserve $C_{24}$.
Thus \eqref{eq:quantitative-tangent-rigidity} yields an orthogonal $U$ and $\tau_{\rm cost}=\sqrt{\operatorname{cost}/8}$ with
\[
 d_H(\widehat P,UP_{\tau_{\rm cost}})
       \le C r_P\sqrt{\varepsilon(P)+r_P}.
\]
Finally, \eqref{eq:sharp-Hausdorff-cubic} and \eqref{eq:intrinsic-quadratic-energy} imply
\[
 \left|\operatorname{cost}-\frac{\sqrt2}{8}r_P^2\right|
       \le C r_P^2\{\varepsilon(P)+r_P\}.
\]
The inequality $|\sqrt a-\sqrt b|\le\sqrt{|a-b|}$ bounds $|\tau_{\rm cost}-r_P/(8c_{\rm vol})|$ by $Cr_P\sqrt{\varepsilon(P)+r_P}$.
Lipschitz dependence on this parameter proves the claimed proximity to the extremal family.
\end{proof}

\begin{proof}[Proof of Corollary~\ref{cor:global-stability}]
Choose a fixed $\varepsilon_0>0$ below the deficit gap in Theorem~\ref{thm:zonotope-gap}, with $\log(1+\varepsilon_0)$ below the threshold of Corollary~\ref{cor:entropy-geometric-control}.
For $\operatorname{def}(P)\le\varepsilon_0$, the principal cone is excluded, so classification and that corollary give
\[
 d_{24}(P)\le C\sqrt{\log(1+\operatorname{def}(P))}
 \le C\sqrt{\operatorname{def}(P)}.
\]
For $\varepsilon_0\le\operatorname{def}(P)\le1$, the area of $\widehat P$ is at most $2^{15/4}$ and Lemma~\ref{lem:degeneracy} bounds its diameter, so enlarging $C$ extends the estimate to this range.
After rotation realizing $d_{24}(P)$, Lemma~\ref{lem:24-cell-local-converse} and Proposition~\ref{prop:local-isoperimetric-cost} give $\operatorname{def}(P)\le Cd_{24}(P)^2$ locally.
Theorem~\ref{thm:sharp-Hausdorff-coefficient} gives the sharp lower coefficient and optimal exponent.
\end{proof}

\section*{Acknowledgments}

\subsection*{Human interactions}

The authors thank Thomas Hales for useful discussions.

\subsection*{LLM interactions}

While the GPT-6 Astra model of ChatGPT was used to explore alternate routes for the concluding step, most attempts failed, and the viable candidates fell short of our original approach. It also provided editorial suggestions and feedback on the manuscript.

\end{document}